\documentclass[hidelinks,onefignum,onetabnum,final,nohypdvips]{siamart250211}

\usepackage{lipsum}
\usepackage{amsfonts}
\usepackage{graphicx}
\usepackage{epstopdf}
\ifpdf
  \DeclareGraphicsExtensions{.eps,.pdf,.png,.jpg}
\else
  \DeclareGraphicsExtensions{.eps}
\fi

\usepackage{lmodern}
\usepackage{multirow}%
\usepackage{amsmath,amssymb,amsfonts,mathtools}%
\usepackage{mathrsfs}%
\usepackage{xcolor}%
\usepackage{textcomp}%
\usepackage{manyfoot}%
\usepackage{booktabs}%
\usepackage{algorithmicx}%
\usepackage{algpseudocode}%
\usepackage{xspace}
\usepackage{url}
\newcommand{\colorurl}[1]{\textcolor{magenta}{\url{#1}}}
\newsiamremark{remark}{Remark}
\newsiamremark{hypothesis}{Hypothesis}
\crefname{hypothesis}{Hypothesis}{Hypotheses}
\newsiamremark{example}{Example}
\newsiamremark{conjecture}{Conjecture}

\headers{Accelerated GD by Concatenation of Stepsize Schedules}{Zehao Zhang and Rujun Jiang}

\title{Accelerated Gradient Descent by Concatenation of Stepsize Schedules
\thanks{Submitted to the editors DATE.\funding{This work is partly supported by the National Key R\&D Program of China under grant 2023YFA1009300, and the Major Program of NSFC (72394360,72394364).}} }
\author{Zehao Zhang\thanks{School of Data Science, Fudan University (\email{zhangzehao24@m.fudan.edu.cn}).}
\and Rujun Jiang\thanks{Corresponding author. School of Data Science, Shanghai Key Laboratory for Contemporary Applied Mathematics, Fudan University (\email{rjjiang@fudan.edu.cn}).}
} 

\usepackage{amsopn}

\DeclareMathOperator{\argmax}{argmax}

\def\cbl{\color{blue}}

\usepackage{soul}

\def\x{{\mathbf x}}
\def\g{{\mathbf g}}
\def\N{{\mathbb N}}
\def\R{{\mathbb R}}

\newcommand*{\CPP}{{Theorem \ref{them-CPP}}}
\newcommand*{\CPD}{Theorem \protect\ref{them-CPD}}
\newcommand*{\CDP}{Theorem \protect\ref{them-CDP}}
\newcommand*{\CC}{Theorems \ref{them-CPP} and \ref{them-CPD}}
\newcommand{\algA}{{\circ}}
\newcommand{\algB}{{\bullet}}
\newcommand{\algC}{{\blacktriangle}}

\newcommand{\dual}{g-bounded\xspace}
\renewcommand{\leqslant}{\leq}
\renewcommand{\geqslant}{\geq}

\renewenvironment{equation*}{\[}{\]\ignorespacesafterend}
\ifpdf
\hypersetup{
  pdftitle={Accelerated Gradient Descent by Concatenation of Stepsize Schedules},
  pdfauthor={Zehao Zhang, and Rujun Jiang}
}
\fi

\renewcommand{\cref}{\Cref}

\def\cbl{{}}
\def\cbll{{}}%\color{blue}}

\begin{document}

\maketitle

% REQUIRED

\begin{abstract}
    This work considers stepsize schedules for gradient descent on smooth convex objectives. We extend the existing literature and propose a unified technique for constructing stepsizes with analytic bounds for an arbitrary number of iterations. This technique constructs new stepsize schedules by concatenating {two stepsize schedules with fewer steps}. Using this approach, we introduce two new families of stepsize schedules, achieving a convergence rate of $O(n^{-\log_2(\sqrt 2+1)})$ with state-of-the-art constants for the objective value and gradient norm of the last iterate, respectively. Furthermore, our analytically derived stepsize schedules either match or surpass the existing best numerically computed stepsize schedules.
\end{abstract}



\section{Introduction}
We consider the {memoryless fixed-step gradient descent method} for minimizing a convex and $L$-smooth function \( f: \mathbb{R}^d \to \mathbb{R} \),  where $L$-smoothness refers to the gradient of $f$ being $L$-Lipschitz continuous. Starting from an initial point \( \mathbf{x}_0 \in \mathbb{R}^d \), the gradient descent method follows the iterative update rule:
\begin{equation}\label{eq:gd}
\mathbf{x}_{i+1} = \mathbf{x}_i - \frac{h_i}{L} \nabla f(\mathbf{x}_i), \quad i = 0, 1, 2, \dotsc, n-1,
\end{equation}
where \( h_i \in \mathbb{R} \) represents the stepsize scaled by \( 1/L \). For simplicity, we refer to \( h = [h_0, \dotsc, h_{n-1}]^T \in \mathbb{R}^n \) as the stepsize schedule (SS), omitting the scaling by \( 1/L \).

For the class of functions under consideration, the textbook result states that using constant stepsizes guarantees an  $O(1/n)$ convergence rate. However, there has been growing interest in developing tighter convergence analyses for gradient descent.
{Drori and Teboulle~\cite{drori2014performance} introduced the performance estimation problem (PEP) framework, and established a tight bound for constant-stepsize gradient descent.  Taylor et al.~\cite{taylor2017smooth} later presented necessary and sufficient conditions for smooth strongly convex interpolation, and showed that the PEP can characterize the exact worst-case performance of any fixed-step first-order algorithm for smooth strongly convex unconstrained optimization.}
% ., by deriving closed-form necessary and sufficient conditions for smooth convex interpolation.} 
Building on this, Das Gupta et al.~\cite{gupta2024branch} recently proposed a branch and bound method to solve the PEP framework to compute the optimal SSs for fastest rates in worst-case performance. Their numerical results suggest that an accelerated rate of $0.156/n^{1.178}$  may be achievable with improved (non-constant) SSs.
The PEP framework, together with the interpolation conditions, has also facilitated the development of more efficient analytical SSs.
Teboulle and Vaisbourd~\cite{teboulle2023elementary}  proposed a dynamic stepsize sequence that achieves better convergence rate compared to the standard textbook result. %
Grimmer~\cite{grimmer2024provably} investigated periodic SSs, and suggested the potential for acceleration with this approach.

Although complexity bounds for SSs can be established using PEP, the reformulation of the problem through SDP is rather intricate, making it challenging to derive tight analytical SSs.
Very recently, two concurrent papers, one by Altschuler and Parrilo~\cite{silver2} and another by Grimmer et al.~\cite{grimmer2023longsteps}, confirmed that accelerated rates of $O(n^{-\varrho})$ and $O(n^{-1.0564})$, respectively, are possible for the worst-case performance. Here we use the convention $\varrho = \log_2(\sqrt{2} + 1) \approx 1.2716$. The rate of~\cite{silver2} was later improved by Grimmer et al.~\cite{grimmer2024acceleratedobjective} by a constant factor. {The works of~\cite{silver2} and~\cite{grimmer2024acceleratedobjective} leveraged the dual formulation of the PEP, wherein worst-case bounds were derived via nonnegative linear combinations of interpolation inequalities---a consequence of strong duality in the PEP framework~\cite{taylor2017smooth}.
% goujaud2023fundamental, kim2021ogm-g, kim2016optimized, lie21}.
}
% The primary techniques in~\cite{silver2} and~\cite{grimmer2024acceleratedobjective} involve expressing the complexity bound as linear combinations of interpolation inequalities for smooth convex functions.

{\cbl
It is important to note that the methods considered in this paper are restricted to {memoryless fixed-step gradient descent} methods of the form~\eqref{eq:gd}, where the SS is determined in advance and does not depend on the objective function $f$, and each update uses only the current gradient $\nabla f(\mathbf{x}_i)$. This excludes more general fixed-step first-order schemes that combine past gradients, as well as methods that introduce additional dynamics or adapt stepsizes based on observed information. For instance, Nesterov's Accelerated Gradient~\cite{nesterov1983method} uses momentum to attain an $O(1/n^2)$ worst-case rate. Its variants, such as the Optimized Gradient Method (OGM)~\cite{kim2016optimized,kim2017ogm} for the objective value and OGM-G~\cite{kim2021ogm-g} for the gradient norm, further improve the respective worst-case constants.  Similarly, adaptive strategies, e.g., Polyak's stepsize~\cite{polyak1987introduction}, the Barzilai–Borwein method~\cite{barzilai1988two}, and other recent adaptive gradient methods~\cite{malitsky2019adaptive, malitsky2024adaptive,zhou2025adabb}, adjust stepsizes based on observed gradients and thus fall outside the fixed-step framework. Although the best-known convergence rate for memoryless fixed-step gradient descent is of order $O(n^{-\varrho})$ and thus slower than the $O(1/n^2)$ rate achievable by momentum methods, determining the fundamental limit of acceleration attainable solely via stepsize scheduling remains an important and interesting theoretical problem.
}
% One drawback of the approaches in~\cite{silver2,grimmer2024acceleratedobjective} is their reliance on extensive computations, often requiring symbolic computation to verify equations. Moreover, their SSs achieving the $O(n^{-\varrho})$ convergence rate are limited to iterations where $n = 2^l-1$ {for integer $l$}, making them inapplicable for general iteration counts.

{% Another intriguing but unanswered question in the field is the determination of optimal SSs, particularly for larger $n$.
A closely related question to determining the ultimate memoryless fixed-step acceleration limit is the construction and determination of optimal SSs, particularly for larger $n$.} Previous studies on optimal SSs~\cite{altschuler2018greed, daccache2019performance, eloi2022worst} were limited to very small $n$. For instance, Daccache~\cite{daccache2019performance} determined the optimal stepsize for $n = 2$ to be $[\sqrt 2,(3+\sqrt{9+8\sqrt{2}})/4]^T$. Numerical results from Das Gupta et al.~\cite{gupta2024branch} suggested optimal SSs of $[1.414,2.414,1.5]^T$ for $n =3$ and $[1.414,1.601,3.005,1.5]^T$ for $n = 4$. These computed optimal SSs showed no clear indication of simple closed-form expressions for general $n$.  This gap between known accelerated rates and truly optimal schedules remains a key area of ongoing research.

Upon closer examination of recent works~\cite{teboulle2023elementary, silver2, grimmer2024acceleratedobjective, rotaru2024exact}, {\cbll it can be seen that their SS constructions follow a common pattern, which we refer to as concatenation.} This methodology involves combining two SSs (say, $h_a\in\R^{n_1}$ and   $h_b\in\R^{n_2}$) to create a longer one, i.e.,
\begin{equation*}\label{intro-concat}
    h_c =[h_a^T,\alpha,h_b^T]^T\in\mathbb R^{n_1+n_2+1},
\end{equation*}
where $\alpha$ is some particular constant depending on $h_a$ and   $h_b$. Although this approach has been implicitly used in~\cite{silver2, grimmer2024acceleratedobjective}, its full  potential has not been thoroughly explored. In this work, we formalize this intuition by establishing a series of theorems, using different types of $h_a$ and $h_b$, and various choices of $\alpha$. These theorems not only provide a theoretical foundation for constructing efficient SSs, but also offer insights that may lead us closer to determining truly optimal schedules in smooth convex optimization.
We propose three special classes of SSs with certain convergence guarantees, constructed through simple concatenation. Based on this, we propose two families of SS: $h_\algB^{(n)}$ for reducing the objective gap, and $h_\algC^{(n)}$ for the gradient norm. Both achieve the state-of-the-art asymptotic rate of $(0.4239 + o(1)) \cdot n^{-\varrho}$, offering better constant dependence than the existing best results in~\cite{grimmer2024acceleratedobjective}. {Notably}, $h_\algB^{(n)}$ not only matches but also outperforms the optimal numerical SSs in~\cite{gupta2024branch}, suggesting its potential optimality in non-asymptotic scenarios.
As a byproduct, we use the concatenation techniques to generate dynamic stepsize sequences
that exhibit a nearly periodic pattern with $O(1/n)$ convergence rate, and allow for arbitrary constant  improvement depending on the length of the ``period''.

The paper is organized as follows. In~\cref{sect:dominance}, we present preliminaries for gradient descent and introduce two new notions for SSs. In~\cref{sect:concat}, we establish two concatenation techniques.
Based on the techniques, in~\cref{sect:applications} we propose a family of SSs $h_\algB^{(n)}$, and analyze their asymptotic convergence.
In~\cref{sect:H-duality}, we introduce another concatenation technique to construct SSs $h_\algC^{(n)}$  and evaluate their performance on the gradient norm.
Then in~\cref{sec:experiment}, we present {comparisons} to demonstrate the advantages of our proposed SSs.
Finally, we conclude the paper in~\cref{sec:con}.

\paragraph{Notation}Let $\mathbb N=\{1,2,\dots\}$ and $\mathbb N_0=\{0,1,2,\dots\}$.
 Let $\R^n$ denote the $n$-dimensional real space and $\R_+^n$ (resp. $\R_-^{n}$) denote the nonnegative (resp. nonpositive) orthant. By convention, when $n = 0$, both $\R^n$ and $\R_+^n$ (resp. $\R_-^{n}$) are considered to contain the zero-dimensional vector (i.e., the empty vector). Unless otherwise specified, SSs can have zero length, i.e., when we write $h\in \R^n$, we allow $n\in \N_0$.
For an empty sum, such as $\sum_{i=0}^{-1}$, the summation is defined to be zero.
Let $\mathcal C^{1,1}_L$ be the class of convex, $L$-smooth (i.e., $\nabla f(x)$ is $L$-Lipschitz) functions $f(\mathbf x): \mathbb R^d\to\mathbb R$, where $L>0$ is a fixed constant.
We denote the global minimizer of $f$ by $\mathbf x_\star$ and the corresponding minimum value by $f_\star= f(\mathbf x_\star)$.
%Given an initial point $\mathbf x_0 \in \mathbb R^d$ and an SS $h =[h_0,\dotsc,h_{n-1}]^T\in\mathbb R^n$, the gradient descent follows the update rule in~\eqref{eq:gd}.
We define $\mathbf g_i = \nabla f(\mathbf x_i)$ and $f_i = f(\mathbf x_i)$ to be the gradient and the function value at $\mathbf x_i$, respectively.
We frequently use the convention $G=[\mathbf g_0,\dotsc,\mathbf g_n]\in\R^{d\times (n+1)}$.
%Given an SS $h=[h_0,\dotsc,h_{n-1}]^T$, let $\widetilde h$ denote the reverse of $h$, i.e., $\widetilde h = [h_{n-1},\dotsc,h_0]^T$.
%The notation $\langle \cdot, \cdot\rangle$ is used to denote the inner product.
We use the convention $\varrho = \log_2(\sqrt{2} + 1) \approx 1.2716$.
\section{Two New Notions on SSs}\label{sect:dominance}
\subsection{Preliminaries}
Our work is based on the following well-known fact of smooth convex functions, whose proof can be referred to, for example,~\cite[Theorem 5.8]{beck2017first}.
\begin{lemma}\label{lemma:Qij0} For a function $f\in\mathcal C^{1,1}_L $, it holds that
\begin{equation*}
    f(\mathbf x) -f(\mathbf y)+\langle\nabla f(\mathbf x), \mathbf y -\mathbf x\rangle + \frac{1}{2L}\Vert\nabla f(\mathbf x)-\nabla f(\mathbf y)\Vert^2\leqslant 0,~\forall\mathbf x,\mathbf y\in\mathbb R^d.
\end{equation*}
\end{lemma}
Lemma \ref{lemma:Qij0} immediately leads to the following results.%corollary if applied to our gradient descent scenario.

\begin{lemma} \label{lemma:Qij} In our gradient descent setting, it holds that
\begin{subequations}
\begin{align} & Q_{i,j} \triangleq f_i - f_j + \langle \mathbf g_i,\mathbf x_j-\mathbf x_i\rangle+\frac{1}{2L}\Vert\mathbf g_i-\mathbf g_j\Vert^2\leqslant 0,\label{Qij} \\
  &  Q_{i,\star} \triangleq f_i-f_\star - \langle \mathbf g_i,\mathbf x_i-\mathbf x_\star\rangle + \frac{1}{2L}\Vert \mathbf g_i\Vert^2\leqslant 0,~ 
 Q_{\star,i} \triangleq f_\star-f_i  + \frac{1}{2L}\Vert \mathbf g_i\Vert^2\leqslant 0\label{Qij2}
\end{align}
\end{subequations}
% \begin{equation}\label{Qij2}
% \left\{\begin{aligned}
%    & Q_{i,\star} \triangleq f_i-f_\star - \langle \mathbf g_i,\mathbf x_i-\mathbf x_\star\rangle + \frac{1}{2L}\Vert \mathbf g_i\Vert^2\leqslant 0,
% \\ & Q_{\star,i} \triangleq f_\star-f_i  + \frac{1}{2L}\Vert \mathbf g_i\Vert^2\leqslant 0
% \end{aligned}\right.
% \end{equation}
for any indices $i,j\in \{0,1,\dotsc,n\}$ and $(f,\mathbf x_0)\in\mathcal C^{1,1}_L\times \mathbb R^d$.
\end{lemma}

\begin{proof}
Letting $\mathbf x = \mathbf x_i$ and $\mathbf y = \mathbf x_j$ in Lemma \ref{lemma:Qij0} yields~\eqref{Qij}. For~\eqref{Qij2}, recall that $\mathbf g_\star=\nabla f(\mathbf x_\star) =\mathbf 0$ because $\mathbf x_\star$ is the minimizer. Thus, letting $(\mathbf x, \mathbf y) = (\mathbf x_i,\mathbf x_\star)$ yields $Q_{i,\star}\leqslant 0$, while $(\mathbf x, \mathbf y) = (\mathbf x_\star, \mathbf x_i)$ leads to $Q_{\star,i}\leqslant 0$.
\end{proof}
In the remainder of this paper, we will use the convention
\begin{equation}\label{eq:defv}
v =  [Q_{0,\star},\dotsc,Q_{n,\star}]^T\in\R_-^{n+1}.
\end{equation}
The notations $Q_{i,j}$, $Q_{i,\star}$, and $Q_{\star,i}$ introduced in Lemma \ref{lemma:Qij} will be used consistently throughout this article. Additionally, we note a useful inequality for $Q_{i,\star}$ as follows.
\begin{lemma}\label{lemma:Qi*-ineq} For $Q_{i,\star}$ defined in Lemma \ref{lemma:Qij}, it holds that
    \begin{equation*}%\label{Qij3}
        Q_{i,\star}\leqslant  f_j-f_\star -\frac{1}{2L}\Vert \mathbf g_j\Vert^2+\frac1L\langle \mathbf
 g_i,\mathbf g_j-L(\mathbf x_j-\mathbf x_\star)\rangle
    \end{equation*}
for any indices $i,j\in \{0,1,\dotsc,n\}$ and $(f,\mathbf x_0)\in\mathcal C^{1,1}_L\times \mathbb R^d$.
\end{lemma}
\begin{proof}
According to~\cref{lemma:Qij}, we have
% \[\begin{aligned}
% Q_{i,\star}&\overset{\eqref{Qij2}}{=}
% f_i -f_\star-\langle \mathbf g_i,\mathbf x_i-\mathbf x_\star\rangle +\frac{1}{2L}\Vert\mathbf g_i\Vert^2
% \\ &\overset{~\eqref{Qij}}{\leqslant}
% ( f_j-f_\star -\frac{1}{2L}\Vert \mathbf g_j\Vert^2)+\frac 1L\langle\mathbf g_i, \mathbf g_j-L(\mathbf x_j-\mathbf x_\star)\rangle.
% \end{aligned}
% \]
$
Q_{i,\star}\overset{\eqref{Qij2}}{=}
f_i -f_\star-\langle \mathbf g_i,\mathbf x_i-\mathbf x_\star\rangle +\frac{1}{2L}\Vert\mathbf g_i\Vert^2\\
\overset{~\eqref{Qij}}{\leqslant}
( f_j-f_\star -\frac{1}{2L}\Vert \mathbf g_j\Vert^2)+\frac 1L\langle\mathbf g_i, \mathbf g_j-L(\mathbf x_j-\mathbf x_\star)\rangle.
$
\end{proof}

\subsection{New SSs}
%While convergence certificates for SSs can be established through the PEP technique~\cite{drori2014performance, taylor2017smooth},
% Elementary approaches for convergence analysis of  gradient descent involving linear combinations of  $Q_{i,j}$s defined in Lemma \ref{Qij} have been demonstrated~\cite{altschuler2018greed, silver1, silver2, grimmer2024acceleratedobjective}.
Previous elementary methods for the convergence analysis of gradient descent often rely on linear combinations of $Q_{i,j}$s, as defined in Lemma \ref{Qij}~\cite{altschuler2018greed, silver1, silver2, grimmer2024acceleratedobjective}.
However, these methods frequently involve cumbersome summations that may require  symbolic computation, and their convergence rates are restricted to iterations $n=2^l-1,{\cbl~l\in\N}$.
Building upon the works of~\cite{altschuler2018greed, silver1, silver2, grimmer2024acceleratedobjective}, we introduce in this section a vectorized computation technique to streamline these calculations, which we term ``dominance''. We also introduce two classes of SSs related to the dominance property, which will be fundamental in the subsequent analysis.

\begin{definition}[Dominance]\label{def-control}
%	Define the vector $v=[Q_{0,\star},\dotsc,Q_{n,\star}]^T\in\mathbb R^{n+1}$.
We say an SS {$h = [h_0,h_1,\dotsc,h_{n-1}]^T\in\mathbb R_+^n$} and a coefficient vector $u\in\mathbb R_+^{n+1}$ \emph{dominate}  $M\ge0$ if
\begin{equation*}%\label{eq:dominate}
    \frac L2\Vert\mathbf x_0-\mathbf x_\star\Vert^2+\langle u,v\rangle - (\mathbf 1^Tu)(f_n-f_\star)\geqslant M,~\forall (f,\mathbf x_0)\in\mathcal C^{1,1}_L\times \mathbb R^d,
\end{equation*}
where $\mathbf 1$ denotes the vector of all ones, and $v$ is defined in~\eqref{eq:defv}.
\end{definition}

\begin{definition}[Dominant SS]\label{def:dominant} We say an SS $h\in\mathbb R_+^n$, is \emph{dominant} if there exists a vector $u\in\mathbb R_+^{n+1}$ such that $\mathbf 1^Tu =2(\mathbf 1^Th)+1$, and that $h$ and $u$ dominate $\frac L2\Vert \mathbf x_0-\mathbf x_\star -L^{-1}Gu\Vert^2$, recalling that
$G = [\mathbf g_0,\dotsc,\mathbf g_n]\in\mathbb R^{d\times (n+1)}$.
\end{definition}

{The next theorem presents a worst-case bound for dominant stepsize schedules, which involves the Huber function as the worst-case function. We remark that the Huber and quadratic functions are often observed as worst-case functions in the PEP literature, e.g., for the constant stepsize schedule \cite{drori2014performance} and for more general first-order methods \cite{taylor2017smooth, kim2017ogm}.}

\begin{theorem}[Upper bound of dominant SS]\label{them-dominate} If an SS $h\in\mathbb R_+^n$ is dominant, then gradient descent~\eqref{eq:gd} with $h$ has
\begin{equation}\label{eq:boundfordom}
    f_n-f_\star\leqslant \frac{1}{2(\mathbf 1^Th)+1}\cdot \frac L2\Vert \mathbf x_0-\mathbf x_\star\Vert^2.
\end{equation}
The bound attains equality when  $\Vert\mathbf x_0 \Vert = 1$ %\todo{$\mathbf x_0\neq 1$ but use $\mathbf x/\|\mathbf x_0\|$}
and $f$ is the Huber function:
\begin{equation}\label{eq:huber1}
    f(\mathbf x)= \left\{\begin{array}{ll}
    \frac Lw\Vert \mathbf x \Vert-\frac{L}{2w^2}, &  \Vert\mathbf x \Vert\geqslant \frac1w,
    \\ \frac L2\Vert\mathbf x \Vert^2, & \Vert \mathbf x \Vert<\frac1w,\end{array}\right.
\end{equation}
where $w =  2(\mathbf 1^Th)+1$.
\end{theorem}

\begin{proof} Since $h$ is dominant, there exists a vector $u\in\mathbb R_+^{n+1}$ such that $h$ and $u$ dominate $M\triangleq\frac L2\Vert\mathbf x_0 -\mathbf x_\star- L^{-1} Gu\Vert^2$. By definition, it follows that $\mathbf 1^Tu = 2(\mathbf 1^Th)+1$. From~\cref{def-control}, $ \frac L2\Vert\mathbf x_0-\mathbf x_\star\Vert^2+\langle u,v\rangle - (\mathbf 1^Tu)(f_n-f_\star)\geqslant M\geqslant 0$ holds, where $v =  [Q_{0,\star},\dotsc,Q_{n,\star}]^T$. Recalling that $v\leqslant 0$ by Lemma \ref{lemma:Qij},  we have $\frac L2\Vert\mathbf x_0-\mathbf x_\star\Vert^2  - (\mathbf 1^Tu)(f_n-f_\star) \geqslant 0$ as $u\geqslant 0$. Plugging in $\mathbf 1^Tu = 2(\mathbf 1^Th)+1$ yields the desired bound.

Next we show that the bound is attained by the Huber function~\eqref{eq:huber1}. It is easy to verify $f\in\mathcal C^{1,1}_L$~\cite[Theorem 6.60]{beck2017first}, and has minimizer $\mathbf x_\star =\mathbf  0$ and minimum $f_\star=0$. Noting that $\mathbf x_0 $ is a unit vector, by induction we can show that $\mathbf x_i = \mathbf x_0 -\frac1w \sum_{j=0}^{i-1}h_j\mathbf x_0$ for $i=1,2,\dots,n$ and $\mathbf x_n  = \mathbf x_0  - \frac 1w( \mathbf 1^Th) \mathbf x_0 =\frac{w+1}{2w} \mathbf x_0 $. Here we have $\Vert\mathbf x_n \Vert = \frac{w+1}{2w}\geqslant \frac1w$. As a result, the equality in~\eqref{eq:boundfordom} holds because
\begin{equation*}
    f_n-f_\star = f(\mathbf x_n) -0= \frac Lw \cdot \frac{w+1}{2w}-\frac{L}{2w^2} = \frac{L}{2w}=\frac{1}{2(\mathbf 1^Th)+1}\cdot \frac L2\Vert\mathbf x_0-\mathbf x_\star\Vert^2.
\end{equation*}
\end{proof}

%Readers may recognize a similarity between the definition of dominance and PEP techniques.\todo{SDP relaxation of PEP problem}
%The semidefinite programming\todo{dual of the PEP problem?} is about finding a set of nonnegative multipliers $\lambda_{i,j}$ and a maximum $C$ such that $\frac L2\Vert\mathbf x_0-\mathbf x_\star\Vert^2 - C\cdot (f_n-f_\star)+\sum_{i,j}\lambda_{i,j}Q_{i,j}$ {is  a positive semidefinite quadratic form of $\mathbf x_0,\dotsc,\mathbf x_n,\mathbf g_0,\dotsc,\mathbf g_n$, where $i,j\in\{0,1,\dotsc,n,\star\}$. Our dominant SS coincides with such a positive semidefinite quadratic form if it is of rank-one, taking the form $\frac L2\Vert\mathbf x_0-\mathbf x_\star -L^{-1}Gu\Vert^2$, where $G = [\mathbf g_0,\dotsc,\mathbf g_n]$. This low-rank structure is also mentioned in~\cite{grimmer2024strengthened},  where they conjecture a rank-one structure for the case of optimal constant SSs}.

{We remark that there is a relationship between the definition of dominance and the dual SDP formulation of the PEP, rooted in the work of~\cite{drori2014performance,taylor2017smooth} and discussed in detail in~\cite[Chapter 8]{altschuler2018greed}. The dual formulation {\cbll offers a useful perspective} to establish proofs of convergence bounds by finding a set of nonnegative multipliers $\lambda_{i,j}$ and the largest $C$ such that $\frac L2\Vert\mathbf x_0-\mathbf x_\star\Vert^2 - C\cdot (f_n-f_\star)+\sum_{i,j}\lambda_{i,j}Q_{i,j}$  is  a positive semidefinite quadratic form of $\mathbf x_0,\dotsc,\mathbf x_n,\mathbf g_0,\dotsc,\mathbf g_n$, where $i,j\in\{0,1,\dotsc,n,\star\}$.} Our dominant SS further assumes that a positive semidefinite quadratic form has rank one, taking the form $\frac L2\Vert\mathbf x_0-\mathbf x_\star -L^{-1}Gu\Vert^2$, where $G = [\mathbf g_0,\dotsc,\mathbf g_n]$. %The low-rank structure was also mentioned in various literature, e.g.,~\cite{kim2021ogm-g}.
% The low-rank structure and its relationship with the worst-case functions were also mentioned in various works of PEP~\cite{drori2014performance, taylor2017smooth, kim2017ogm, daccache2019performance, eloi2022worst, grimmer2024strengthened}.
% A similar low-rank structure was also mentioned in~\cite{grimmer2024strengthened},  where the authors conjectured a rank-one structure for the case of optimal constant SSs.
{We note that low-rank structures have been repeatedly observed in the PEP and first-order optimization literature (see, e.g.,~\cite{kim2021ogm-g, kim2016optimized, lie21, kim2021accelerated}). More recently, Grimmer et al.~\cite{grimmer2024strengthened} explicitly conjectured a rank-one structure for the case of optimal constant stepsize schedules.
}

We next propose another class of SSs. % $h$, the primitive SSs.

\begin{definition}[Primitive SS]\label{def:primitive} We say an SS $h\in\mathbb R_+^n$ is \emph{primitive} if $h$ and $u=[h^T,0]^T\in\mathbb R_+^{n+1}$ dominate $\frac L2\Vert\mathbf x_n-\mathbf x_\star\Vert^2 +\frac{C}{2L}\Vert\mathbf g_n\Vert^2$, where $C = (\mathbf 1^Th)\cdot (\mathbf 1^Th+1)$.
\end{definition}

We now demonstrate that primitive SSs belong to the class of dominant SSs.
\begin{proposition}\label{them:primitive-are-dominant}
Suppose $h\in\mathbb R_+^n$   is primitive. Then $h$ is dominant. Particularly, {for ${\bar u} = [h^T,\mathbf 1^Th+1]^T\in\mathbb R_+^{n+1}$, $h$ and $\bar u$ dominate} $\frac L2\Vert\mathbf x_n -\mathbf x_\star-(\mathbf 1^Th+1)L^{-1}\mathbf g_n\Vert^2$.
\end{proposition}

\begin{proof} Let $u=[h^T,0]^T\in\mathbb R^{n+1}$. By~\cref{def:primitive}, we have
\begin{equation}\label{eq:prieq}
\frac L2\Vert\mathbf x_0-\mathbf x_\star\Vert^2+\langle u,v\rangle - (\mathbf 1^Tu)(f_n-f_\star)\geqslant \frac L2\Vert\mathbf x_n-\mathbf x_\star\Vert^2 +\frac{C}{2L}\Vert\mathbf g_n\Vert^2,
\end{equation}
where $v$ is defined in~\eqref{eq:defv} and $C = (\mathbf 1^Th)\cdot(\mathbf 1^Th+1)$. %= [Q_{0,\star},\dotsc,Q_{n,\star}]^T\in\mathbb R^{n+1}$.
Let $r=\mathbf 1^Th$ and ${\bar u} = [h^T,r+1]^T$, so $C = r(r+1)$ and $\mathbf 1^T\bar u=\mathbf 1^Tu+r+1$. Then we have
\begin{equation}\label{eq:primitive}\begin{aligned}
  &\quad\ \frac L2\Vert\mathbf x_0-\mathbf x_\star\Vert^2+\langle {\bar u},v\rangle - (\mathbf 1^T{\bar u})(f_n - f_\star)\\
  &\!\!\!\stackrel{\eqref{eq:prieq}}{\geqslant} \left.\frac L2\Vert\mathbf x_n-\mathbf x_\star\Vert^2 +\frac{r(r+1)}{2L}\Vert\mathbf g_n\Vert^2  \right.+(r+1)(Q_{n,\star} - (f_n-f_\star))
    \\
        & \!\!\!\stackrel{(\ref{Qij2})}{=} \left.\frac L2\Vert\mathbf x_n-\mathbf x_\star\Vert^2 +\frac{r(r+1)}{2L}\Vert\mathbf g_n\Vert^2  \right.+(r+1)\left(-\langle \mathbf g_n,\mathbf x_n-\mathbf x_\star\rangle+\frac{1}{2L}\Vert\mathbf g_n\Vert^2\right)
    \\  &=  \frac L2\Vert   \mathbf x_n-\mathbf x_\star\Vert^2-(r+1)\langle \mathbf g_n,\mathbf x_n-\mathbf x_\star\rangle+\frac{(r+1)^2}{2L}\Vert \mathbf g_n\Vert^2
    \\ &=  \frac L2\Vert \mathbf x_n-\mathbf x_\star-(r+1)L^{-1}\mathbf g_n\Vert^2 \geqslant 0.
\end{aligned}
\end{equation}
Inequality~\eqref{eq:primitive} implies that $h$ and $\bar u$ dominate $\frac L2\Vert \mathbf x_n-\mathbf x_\star - (\mathbf 1^Th+1)L^{-1}\mathbf g_n\Vert^2$. We now show that this implies that $h$ is {dominant}. Recalling $G=[\mathbf g_0,\dots,\mathbf g_n]$, we have
\begin{equation*}
    \mathbf x_0  -L^{-1} G\bar u = \mathbf x_0 - \frac{h_0}{L}\mathbf g_0 -\dots - \frac{h_{n-1}}{L}\mathbf g_{n-1} - \frac{\mathbf 1^Th+1}{L}\mathbf g_n = \mathbf x_n  - (\mathbf 1^Th+1)L^{-1}\mathbf g_n,
\end{equation*}
where the last equation follows that $\x_n=\mathbf x_0 - (h_0/L)\mathbf g_0 - \dots - (h_{n-1}/L)\mathbf g_{n-1}$ from the definition of the gradient descent method in~\eqref{eq:gd}. So~\eqref{eq:primitive} indicates $h$ and $\bar u$ dominate $\frac L2\Vert\mathbf x_0-\mathbf x_\star - L^{-1}G\bar u\Vert^2$. Finally, noting that $\mathbf 1^T\bar u= 2(\mathbf 1^Th)+1$, we conclude $h$ is dominant.
\end{proof}

Since every primitive schedule is dominant, we obtain an immediate corollary of~\cref{them-dominate} and~\cref{them:primitive-are-dominant}.

\begin{corollary}\label{coro:bound-primitive} If an SS $h\in\mathbb R_+^n$ is primitive, then gradient descent~\eqref{eq:gd} with $h$ has
%\begin{equation*}
 $   f_n-f_\star\leqslant \frac{1}{2(\mathbf 1^Th)+1}\cdot \frac L2\Vert\mathbf x_0-\mathbf x_\star\Vert^2.$
%\end{equation*}
The equality is attained  by the Huber function~\eqref{eq:huber1}.
\end{corollary}

%Dominant and primitive SSs have convergence bounds of the form $f_n-f_\star\leqslant \frac{1}{2(\mathbf 1^Th)+1}\cdot \frac L2\Vert\mathbf x_0-\mathbf x_\star\Vert^2$, as shown in~\cref{them-dominate} and~\cref{coro:bound-primitive}.
Let us now present some examples of SSs that fall within these categories. The one-step schedule $[\sqrt 2]^T$ is primitive, while the schedule $[3/2]^T$ is dominant; the proofs are provided in Examples \ref{ex:2} and \ref{ex:3} in the next section. In particular, we emphasize that the simplest and most trivial stepsize, the empty stepsize $h= [~]^T\in\mathbb R^0$, is primitive, which forms the basis of our concatenation technique.

\begin{example}\label{example-zero} The zero-iteration stepsize $h=[~]^T\in\mathbb R^0$ is primitive, and thus it is dominant.
To see this, let $u = [h^T,0]^T=[0]^T\in\mathbb R^1$. Following Definitions \ref{def-control} and \ref{def:primitive},  it is clear that
\[\frac L2\Vert\mathbf x_0-\mathbf x_\star\Vert^2+\langle u,v\rangle - (\mathbf 1^Tu)(f_n-f_\star)= \frac L2\Vert\mathbf x_0-\mathbf x_\star\Vert^2+\frac{0(0+1)}{2L}\Vert\mathbf g_0\Vert^2,\] so $h = [~]^T$ is primitive.~\cref{them:primitive-are-dominant} ensures that $h$ is dominant.
\end{example}

\section{Concatenation Techniques}\label{sect:concat}
In this section, we show that special dominant and primitive SSs can be constructed via concatenation of {two specific SSs}.

\begin{theorem}[Concatenation of two primitive SSs (\texttt{ConPP})]\label{them-CPP}
Suppose $h_a\in\mathbb R_+^n$ and $h_b\in\mathbb R_+^{m-n-1}$ are both primitive SSs, where $m\geqslant n+1$.
Then the SS $h_c =[h_a^T,\alpha,h_b^T]^T\in\mathbb R_+^{m}$ is also primitive, where $\alpha =\varphi(\mathbf 1^Th_a,\mathbf 1^Th_b )$ and  $\varphi(\cdot,\cdot)$ is defined by
\begin{equation}\label{eq:def-varphi}
\varphi(x,y) =\frac{-x-y+\sqrt{(x+y+2)^2+4(x+1)(y+1)}}{2}.
\end{equation}
\end{theorem}
Following the above theorem, we denote by $h_c = \Call{ConPP}{h_a,h_b}$ a procedure of \texttt{ConPP} with $h_a$ and $h_b$ as input and $h_c$ as output.

\begin{theorem}[Concatenation of primitive and dominant SSs (\texttt{ConPD})] \label{them-CPD}
Suppose $h_a\in\mathbb R_+^n$ is a primitive SS and $h_d\in\mathbb R_+^{m-n-1}$ is a dominant SS, where $m\geqslant n+1$.
Then the SS $h_e =[h_a^T,\beta,h_d^T]^T\in\mathbb R_+^{m}$ is also dominant, where $\beta =\psi (\mathbf 1^Th_a,\mathbf 1^Th_d)$ and  $\psi(\cdot,\cdot)$ is defined by
\begin{equation}\label{eq:def-psi}
\psi(x,y) =\frac{3-2y+\sqrt{(2y+1)(2y+8x+9)}}{4}.
\end{equation}
\end{theorem}
% Suppose \todo{$n\in\N_0$} and $h_a\in\mathbb R_+^n$ is primitive. Also, suppose another SS  $h_d\in\mathbb R_+^{m-n-1}$ and coefficient vector  $u_d\in\mathbb R_+^{m-n}$ dominate $\frac12\Vert \mathbf x_0 -G_du_d\Vert^2$ where $G_d=[\mathbf g_0,\dotsc,\mathbf g_{m-n-1}]$ and $m\geqslant n+1$.

% Then the SS  $h_e =[h_a^T,\beta,h_d^T]^T\in\mathbb R_+^m$ and coefficient vector $u_e =[h_a^T,\gamma,(\lambda_1-\lambda_2)u_d^T]^T\in\mathbb R_+^{m+1}$ dominate $\frac12\Vert
%  \mathbf x_0 -G_eu_e\Vert^2$, where  $G_e = [\mathbf g_0,\dotsc,\mathbf g_m]$. Parameters $\beta,\gamma,\lambda_1,\lambda_2$ are given by the following equations:
% \begin{equation}\label{eq:def-psi}
% \left\{\begin{aligned}
%     \gamma &= \mathbf 1^Th_a+2,
%     \\ \lambda_2 &= \frac{ 2(\mathbf 1^Th_a+1)}{\mathbf 1^Tu_d},
%     \\ \lambda_1&=\lambda_2+\frac12+\sqrt{\lambda_2+\frac14},
%     \\ \beta &= \frac{\gamma(\lambda_1-\lambda_2)+\lambda_2}{\lambda_1}.
%     \end{aligned}\right.
% \end{equation}\todo{Reduce four paras to one? Like Theorem 5.2?}
% \end{theorem}
We denote by $h_e=\Call{ConPD}{h_a,h_d}$ a procedure of \texttt{ConPD} with $h_a$ and $h_d$ as input and $h_e$ as output.

In brief, \texttt{ConPP} offers a powerful tool to construct primitive SSs recursively. Thereafter, such primitive SSs can be fed into \texttt{ConPD}, yielding a class of dominant SSs.  We remark that in~\cref{them-CPP} (resp. \ref{them-CPD}), $h_a$ and/or $h_b$ ($h_d$) can be empty.
We provide two examples to  better demonstrate the applications of the theorems.

\begin{example}\label{ex:2}
 The stepsize $h= [\sqrt 2]^T\in\mathbb R^1$ is primitive, so gradient descent~\eqref{eq:gd} with $h$ has bound $f_1-f_\star\leqslant \frac{1}{2\sqrt 2+1}\cdot \frac L2\Vert\mathbf x_0-\mathbf x_\star\Vert^2$.
 To see this, recall in Example \ref{example-zero} we have shown that $[~]^T\in\mathbb R^0$ is primitive. Applying~\CPP \  on  $h_a = h_b=[~]^T$, the intermediate $\alpha$ is given by $\alpha = \varphi(0,0)= \sqrt 2.$
%\begin{equation*}
%    \alpha = \varphi(0,0) = \frac{-0-0+\sqrt{(0+0+2)^2+4\cdot (0+1)\cdot(0+1)}}{2} = \sqrt 2.
%\end{equation*}
Hence $h_c = \Call{ConPP}{h_a,h_b}=[h_a^T,\alpha,h_b^T]^T = [\sqrt 2]^T$ is primitive, and~\cref{coro:bound-primitive} gives the desired bound.
\end{example}

\begin{example}\label{ex:3}
The stepsize $h=[3/2]^T\in\mathbb R^1$ is dominant, so gradient
 descent~\eqref{eq:gd} with $h$ has bound $f_1-f_\star\leqslant \frac14\cdot \frac L2\Vert\mathbf x_0-\mathbf x_\star\Vert^2$. % by~\cref{them-dominate}.
To see this, recall in Example \ref{example-zero}, $[~]^T\in\mathbb R^0$ is primitive and dominant. Taking $h_a = h_d=[~]^T$ in \CPD, the intermediate $\beta$ is given by $\beta =\psi (0,0) = \frac{3}{2}.$
% \begin{equation*}
%     \beta =\psi (0,0)= \frac{3-2\times 0+\sqrt{(2\times 0 +1)\cdot (2\times 0+8\times 0+9)}}{4} = \frac{3}{2}.
% \end{equation*}
Hence $h_e=\Call{ConPD}{h_a,h_d}=[h_a^T,\beta,h_d^T]^T = [3/2]^T$ is dominant, and the desired bound follows from~\cref{them-dominate}.
\end{example}

\subsection{Proof of~\cref{them-CPP}}

%We begin with a lemma asserting that $1<\alpha<\mathbf 1^Th_b+2$ in~\cref{them-CPP}.
To begin, we have the following lemma that shows sufficient decrease of objective values for multiple steps associated with a primitive SS.
\begin{lemma}[Sufficient decrease for multiple steps]\label{lemma:sd} Suppose $h_a\in\mathbb R^n$ is any SS (not necessarily primitive), and $h_b\in\mathbb R_+^{m-n-1}$ is a primitive SS, where $m\geqslant n+1$. For any $\alpha \in [1,\mathbf 1^Th_b+2)$, the concatenated SS $h_c = [h_a^T,\alpha,h_b^T]^T\in\mathbb R^m$ satisfies that
\begin{equation}\label{eq:sd}
f_n - f_m\geqslant \frac{\mathbf 1^Th_b+3\alpha-2\alpha^2}{\mathbf 1^Th_b+2-\alpha}\cdot \frac{1}{2L}\Vert\mathbf g_n\Vert^2+ \frac{2(\mathbf 1^Th_b)^2+3(\mathbf 1^Th_b)+\alpha}{\mathbf 1^Th_b+2-\alpha}\cdot \frac{1}{2L}\Vert\mathbf g_m\Vert^2.
\end{equation}
\end{lemma}

\cref{lemma:sd} generalizes~\cite[Lemma 5]{teboulle2023elementary}, which shows that an SS $h=[h_a^T,\alpha]^T\in\mathbb R^{n+1}$ with $1\leqslant\alpha<2$ satisfies
\begin{equation}\label{eq:lem5tv}
    f_n - f_{n+1}\geqslant \frac{3\alpha-2\alpha^2}{2-\alpha}\cdot \frac{1}{2L}\Vert\mathbf g_n\Vert^2+ \frac{\alpha}{2-\alpha}\cdot \frac{1}{2L}\Vert\mathbf g_{n+1}\Vert^2.
\end{equation}
In fact, when taking $h_b=[~]^T\in\mathbb R^0$ to be the zero-iteration SS from~\cref{example-zero}, our~\cref{lemma:sd} reduces to~\eqref{eq:lem5tv}.
% and a multistep variant in~\cite[Corollary 4.14]{rotaru2024exact}.
% Our lemma extends the results to the class of primitive schedules, connecting gradient norms of long-gap iterations.
%\todo{}Although the coefficients of $\Vert\mathbf g_n\Vert^2$ and $\Vert\mathbf g_m\Vert^2$ in~\eqref{eq:sd} seem to differ from~\eqref{eq:CTP-trans3}, we will show that they are actually equal with the $\alpha$ given by~\eqref{eq:def-varphi}.

\begin{proof}[Proof of Lemma \ref{lemma:sd}]

According to~\cref{def:primitive}, since $h_b$ is primitive and it starts from the $(n+1)$-th step of  $h_c$, we have
\begin{equation}\label{eq:sd-1}
 \frac L2\Vert\mathbf x_{n+1}-\mathbf x_\star\Vert^2+\langle u_b,v_b\rangle - (\mathbf 1^Tu_b)(f_m-f_\star)\geqslant \frac{L}{2}\Vert\mathbf x_m-\mathbf x_\star\Vert^2+\frac{C_b}{2L}\Vert\mathbf g_m\Vert^2,
\end{equation}
where $u_b=[h_b^T,0]^T\in\mathbb R_+^{m-n}$, $v_b =[Q_{n+1,\star},\dotsc,Q_{m,\star}]^T\in\mathbb R^{m-n}$ and $C_b = (\mathbf 1^Th_b)\cdot (\mathbf 1^Th_b+1)$.

Letting $j=n$ or $j=m$ and $i=n+1,\dots,m$ in~\cref{lemma:Qi*-ineq}, by $\langle u_b,v_b\rangle=\langle u_b, [Q_{n+1,\star},\dotsc,Q_{m,\star}]^T\rangle$ and $u_b\geqslant 0$ we have that
\begin{equation}\label{eq:sd-2}\left\{
\begin{aligned}
  &   \langle u_b, v_b\rangle \leqslant \langle u_b, \mathbf 1(f_n-f_\star -\frac{1}{2L}\Vert \mathbf g_n\Vert^2)\rangle+\frac1L\langle G_bu_b, \mathbf g_n -L(\mathbf x_n-\mathbf x_\star)\rangle,
    \\ &  \langle u_b, v_b\rangle \leqslant \langle u_b, \mathbf 1(f_m -f_\star-\frac{1}{2L}\Vert \mathbf g_m\Vert^2)\rangle+\frac1L\langle G_bu_b, \mathbf g_m -L (\mathbf x_m-\mathbf x_\star)\rangle,
    \end{aligned}\right.
\end{equation}
where $G_b=[\mathbf g_{n+1},\dotsc,\mathbf g_m]$. Noting that $u_b=[h_b^T,0]^T = [h_{n+1},\dotsc, h_{m-1},0]^T$, we have the identity $\frac1L G_bu_b =\frac{ h_{n+1}}{L}\mathbf g_{n+1}+\dotsc +\frac{h_{m-1}}{L}\mathbf g_{m-1}=\mathbf x_{n+1} - \mathbf x_m$
by the definition of the gradient descent method~\eqref{eq:gd}. Thus, summing up~\eqref{eq:sd-2} gives
\begin{equation}\label{eq:sd-3}
\begin{aligned}
    2\langle u_b,v_b\rangle &\leqslant (\mathbf 1^Tu_b) (f_n+f_m -2f_\star- \frac{1}{2L}\Vert \mathbf g_n\Vert^2-\frac{1}{2L}\Vert \mathbf g_m\Vert^2 )
    \\ &\quad\quad\
 +\langle \mathbf x_{n+1} - \mathbf x_m,\mathbf g_n+\mathbf g_m-L(\mathbf x_n+\mathbf x_m-2\mathbf x_\star)\rangle.
\end{aligned}
\end{equation}
Subtracting~\eqref{eq:sd-3} from twice of~\eqref{eq:sd-1}  and eliminating $\langle u_b,v_b\rangle$, we obtain
\begin{equation}\small \label{eq:sd-4}\begin{aligned}
    0 &\geqslant L \Vert\mathbf  x_m-\mathbf x_\star\Vert^2-L\Vert\mathbf   x_{n+1}-\mathbf x_\star\Vert^2+\frac{C_b}{L}\Vert\mathbf g_m\Vert^2
    -(\mathbf 1^Tu_b) (f_n-f_m  -  \frac{1}{2L}\Vert \mathbf g_n\Vert^2- \frac{1}{2L}\Vert \mathbf g_m\Vert^2 )\\ &\quad\quad\  -\langle \mathbf x_{n+1}-\mathbf x_m,\mathbf g_n+\mathbf g_m-L(\mathbf x_n+\mathbf x_m-2\mathbf x_\star)\rangle.
    \end{aligned}
\end{equation}
Note that $\alpha$ is the scaled stepsize from $\mathbf x_n$ to $\mathbf x_{n+1}$, so $\mathbf x_{n} = \mathbf x_{n+1}+\frac{\alpha}{L} \mathbf g_n$.  Consequently,~\eqref{eq:sd-4} is simplified to
\begin{equation}\small \label{eq:sd-5}\begin{aligned}
    0 &\geqslant L\Vert  \mathbf x_m-\mathbf x_\star\Vert^2-L\Vert  \mathbf x_{n+1}-\mathbf x_\star\Vert^2+ \frac{C_b}{L}\Vert\mathbf g_m\Vert^2
   -(\mathbf 1^Tu_b) (f_n-f_m - \frac{1}{2L}\Vert \mathbf g_n\Vert^2-\frac{1}{2L}\Vert \mathbf g_m\Vert^2 )
   \\ & \quad\quad\  -\langle \mathbf x_{n+1}-\mathbf x_m,(1-\alpha)\mathbf g_n+\mathbf g_m-L(\mathbf x_{n+1}+\mathbf x_m-2\mathbf x_\star)\rangle
     \\ &=  \frac{C_b}{L}\Vert\mathbf g_m\Vert^2
    -(\mathbf 1^Tu_b)(f_n-f_m-\frac{1}{2L}\Vert\mathbf g_n\Vert^2-\frac{1}{2L}\Vert \mathbf g_m\Vert^2 )-\langle \mathbf x_{n+1}-\mathbf x_m,(1-\alpha)\mathbf g_n + \mathbf g_m \rangle.
    \end{aligned}
\end{equation}
Using $Q_{n,m},Q_{m,n}\leqslant 0$ by Lemma \ref{lemma:Qij} and $\alpha\geqslant 1$ by assumption, we obtain that
\begin{equation}\begin{aligned}\label{eq:sd-6}
0&\geqslant (\alpha-1)Q_{n,m}+Q_{m,n}\\ &
=(\alpha-1) \left(f_n  -f_m +\langle \mathbf g_n, \mathbf x_m - \mathbf x_n\rangle+\frac{1}{2L}\Vert  \mathbf g_m -  \mathbf g_n\Vert^2\right)
\\ &\quad\ +\left(f_m  -f_n +\langle \mathbf g_m, \mathbf x_n - \mathbf x_m\rangle+\frac{1}{2L}\Vert  \mathbf g_m -  \mathbf g_n\Vert^2\right)
\\ &=(\alpha-2)(f_n-f_m)+\langle \mathbf g_m-(\alpha-1)\mathbf g_n, \mathbf x_n - \mathbf x_m\rangle+\frac{\alpha}{2L}\Vert \mathbf g_m-\mathbf g_n\Vert^2.
    \end{aligned}
\end{equation}
Summing up~\eqref{eq:sd-5} and~\eqref{eq:sd-6} directly and recalling that $\mathbf x_n-\mathbf x_{n+1}=\frac{\alpha}{L}\mathbf g_n$, we obtain
\begin{equation}\label{eq:sd-7}\begin{aligned}
    0&\geqslant  \frac{C_b}{L}\Vert\mathbf g_m\Vert^2
    -(\mathbf 1^Tu_b)(f_n-f_m-\frac{1}{2L}\Vert\mathbf g_n\Vert^2-\frac{1}{2L}\Vert \mathbf g_m\Vert^2 )
    \\ & \quad\ + \langle \mathbf g_m -(\alpha - 1)\mathbf g_n,\frac{\alpha}{L}\mathbf g_n\rangle + (\alpha - 2)(f_n - f_m) +\frac{\alpha}{2L}\Vert\mathbf g_m - \mathbf g_n\Vert^2
    \\ &= \left(\alpha - 2-\mathbf 1^Tu_b\right)(f_n-f_m) +\frac{ \mathbf 1^Tu_b+3\alpha-2\alpha^2}{2L}\Vert\mathbf g_n\Vert^2+ \frac{2C_b+\mathbf 1^Tu_b+\alpha}{2L}\Vert\mathbf g_m\Vert^2.
    \end{aligned}
\end{equation}
Finally, recalling $u_b=[h_b^T,0]^T$, we have $\mathbf 1^Tu_b=\mathbf 1^Th_b$. Also, note that we assumed $\mathbf 1^Tu_b+2-\alpha>0$. Therefore~\eqref{eq:sd-7} is equivalent to~\eqref{eq:sd}. The proof is complete.
\end{proof}

Now we are ready to prove~\CPP.
\begin{proof}[Proof of~\CPP]

Recall that both schedules $h_a$ and $h_b$ are primitive as assumed. Since we adopt $h_a$ in the first $n$ steps and $h_b$ in the last $(m-n-1)$ steps, denoting $u_a =[h_a^T,0]^T$ and $u_b =[h_b^T,0]^T$, Definition \ref{def:primitive} implies
\begin{subequations}
\begin{align}
     & \frac L2\Vert\mathbf x_0-\mathbf
 x_\star\Vert^2+\langle u_a,v_a\rangle - (\mathbf 1^Tu_a)(f_n-f_\star)\geqslant \frac{L}{2}\Vert\mathbf x_n-\mathbf
 x_\star\Vert^2+\frac{C_a}{2L}\Vert\mathbf g_n\Vert^2,\label{eq:CTP11}
    \\ &
   \frac L2\Vert\mathbf x_{n+1}-\mathbf
 x_\star\Vert^2+\langle u_b,v_b\rangle - (\mathbf 1^Tu_b)(f_m-f_\star)\geqslant \frac{L}{2}\Vert\mathbf x_m-\mathbf
 x_\star\Vert^2+\frac{C_b}{2L}\Vert\mathbf g_m\Vert^2,\label{eq:CTP12}
 \end{align}
\end{subequations}
where  $C_a = (\mathbf 1^Th_a)\cdot (\mathbf 1^Th_a+1)$, $C_b = (\mathbf 1^Th_b)\cdot (\mathbf 1^Th_b+1)$, $v_a = [Q_{0,\star},\dotsc,Q_{n,\star}]^T$ and $v_b = [Q_{n+1,\star},\dotsc,Q_{m,\star}]^T$.
Moreover, by the definition of $Q_{n,\star}$ in~\eqref{Qij2}, we have the identity
\begin{equation}\label{eq:CTP2}
\alpha(Q_{n,\star} -  (f_n-f_\star)) = \alpha(-\langle \mathbf g_n,\mathbf x_n-\mathbf
 x_\star\rangle +\frac{1}{2L}\Vert\mathbf  g_n\Vert^2 ).
\end{equation}

% It suffices to prove
% \begin{equation}\label{eq:CTP-trans1}
% \begin{aligned}
%   &  -\frac L2\Vert\mathbf x_{n+1}-\mathbf
%  x_\star\Vert^2 +(\mathbf 1^Tu_a)(f_n-f_m)
%     \\ &\quad\ \geqslant-\frac L2\Vert\mathbf x_n-\mathbf
%  x_\star\Vert^2-\frac{C_a+\alpha}{2L}\Vert\mathbf g_n\Vert^2+\frac{C_c-C_b}{2L}\Vert\mathbf g_m\Vert^2 -\alpha (f_n-f_m-\langle\mathbf g_n,\mathbf x_n-\mathbf
%  x_\star\rangle ),
% \end{aligned}
% \end{equation}
% because~\eqref{eq:CTP11}+\eqref{eq:CTP12}+\eqref{eq:CTP2}+\eqref{eq:CTP-trans1}=\eqref{eq:CTP-target2},

% Noting that $\alpha$ is the scaled stepsize from $\mathbf x_n$ to $\mathbf x_{n+1}$, we have $\mathbf x_{n+1} =\mathbf x_n -\frac{\alpha}{L}\mathbf g_n$. Eliminating $\mathbf x_{n+1}$ in~\eqref{eq:CTP-trans1}, we thus need to show
% \begin{equation*}%\label{eq:CTP-trans2}
%   (\mathbf 1^Tu_a)(f_n-f_m)  \geqslant \frac{\alpha^2-\alpha-C_a}{2L}\Vert\mathbf g_n\Vert^2+\frac{C_c-C_b}{2L}\Vert\mathbf g_m\Vert^2 -\alpha (f_n-f_m),
% \end{equation*}
% or, equivalently,
% \begin{equation}\label{eq:CTP-trans3}
%   f_n-f_m  \geqslant \frac{\alpha^2-\alpha-C_a}{\mathbf 1^Tu_a+\alpha}\cdot\frac{1}{2L}\Vert\mathbf g_n\Vert^2+\frac{C_c-C_b}{\mathbf 1^Tu_a+\alpha}\cdot\frac{1}{2L}\Vert\mathbf g_m\Vert^2.
% \end{equation}

% % Inequality~\eqref{eq:CTP-trans3} is only concerning the trailing $(m-n)$ steps.

 Let $x =\mathbf 1^Th_a$, $y=\mathbf 1^Th_b$ for notational simplicity. Then $C_a = x(x+1)$ and $C_b=y(y+1)$.
 Recall $\alpha=\varphi(\mathbf 1^Th_a,\mathbf 1^Th_b)$ as in~\eqref{eq:def-varphi}. Lemma \ref{lemma:CTP} implies that $\varphi(\mathbf 1^Th_a,\mathbf 1^Th_b)\in (1,\mathbf 1^Th_b+2)$. Hence Lemma \ref{lemma:sd} applies here, leading to
 \begin{equation}\label{eq:CTP3}
 f_n -f_m\geqslant \frac{y+3\alpha-2\alpha^2}{y+2-\alpha}\cdot \frac{1}{2L}\Vert\mathbf g_n\Vert^2+\frac{2y^2+3y+\alpha}{y+2-\alpha}\cdot \frac{1}{2L}\Vert\mathbf g_m\Vert^2.
 \end{equation}
 Noting that $\alpha=\varphi(x,y)$ is the root of the quadratic equation $\alpha^2+(x+y)\alpha - (xy+2x+2y+2)=0$, we have the identity
 \begin{equation}\label{eq:CTP-verify1}
 \begin{aligned}
&\frac{\alpha^2 - \alpha - C_a}{x+\alpha}-\frac{y+3\alpha-2\alpha^2}{y+2-\alpha} = \frac{\alpha^2 - \alpha-x(x+1)}{x+\alpha}-\frac{y+3\alpha-2\alpha^2}{y+2-\alpha}
\\
  =&~\alpha -x-1-\frac{y+3\alpha-2\alpha^2}{y+2-\alpha}
  = \frac{\alpha^2+(x+y)\alpha - (xy+2x+2y+2)}{y+2-\alpha}
  =0.
\end{aligned}
 \end{equation}
 Also, recalling $h_c =[h_a^T,\alpha,h_b^T]^T$ gives $\mathbf 1^Th_c= x+\alpha+y$. Letting  $C_c = (\mathbf 1^Th_c)\cdot (\mathbf 1^Th_c+1)= (x+\alpha+y)(x+\alpha+y+1)$, we have
\begin{equation}\label{eq:CTP-verify2}
 \begin{aligned}
&\frac{C_c-C_b}{x+\alpha}-\frac{2y^2+3y+\alpha}{y+2-\alpha} = \frac{ (x+\alpha+y)(x+\alpha+y+1)-y(y+1)}{x+\alpha}-\frac{2y^2+3y+\alpha}{y+2-\alpha}
\\
=&~(x+\alpha+2y+1)-\frac{2y^2+3y+\alpha}{y+2-\alpha}
 = \frac{-\alpha^2-(x+y)\alpha + (xy+2x+2y+2)}{y+2-\alpha} =0.
\end{aligned}
 \end{equation}
Using~\eqref{eq:CTP-verify1} and~\eqref{eq:CTP-verify2}, we have that~\cref{eq:CTP3} is equivalent to
\begin{equation}\label{eq:CTP4}
    (x+\alpha)(f_n-f_m)\geqslant (\alpha^2 - \alpha - C_a)\cdot \frac{1}{2L}\Vert\mathbf g_n\Vert^2+(C_c-C_b)\cdot \frac{1}{2L}\Vert\mathbf g_m\Vert^2.
\end{equation}

Now we compute the sum $\eqref{eq:CTP11}+\eqref{eq:CTP12}+\eqref{eq:CTP2}+\eqref{eq:CTP4}$ to obtain
\begin{equation}\small\label{eq:CTP5}\begin{aligned}
    \frac L2\Vert\mathbf x_0 - \mathbf x_\star\Vert^2&+(\langle u_a,v_a\rangle+\alpha Q_{n,\star}+\langle u_b,v_b\rangle)-(\mathbf 1^Tu_a+\alpha+\mathbf 1^Tu_b)(f_m-f_\star)+\frac L2\Vert\mathbf x_{n+1}-\mathbf x_\star\Vert^2
   \\ & \geqslant \frac L2\Vert\mathbf x_n-\mathbf x_\star\Vert^2+\frac L2\Vert\mathbf x_m-\mathbf x_\star\Vert^2-\alpha\langle \mathbf g_n,\mathbf x_n-\mathbf x_\star\rangle+\frac{\alpha^2}{2L}\Vert\mathbf g_n\Vert^2+\frac{C_c}{2L}\Vert\mathbf g_m\Vert^2.
\end{aligned}
\end{equation}
Define $u_c = [h_c^T,0]^T = [h_a^T,\alpha,h_b^T,0]^T\in\mathbb R_+^{m+1}$ and $v_c = [Q_{0,\star},\dotsc,Q_{m,\star}]^T$. Then we have $\langle u_c,v_c \rangle =\langle u_a,v_a\rangle+\alpha Q_{n,\star}+\langle u_b,v_b\rangle$ and $\mathbf 1^Tu_c = \mathbf 1^Tu_a+\alpha+\mathbf 1^Tu_b$. Also, since $\mathbf x_{n+1} = \mathbf x_n - \frac{\alpha}{L}\mathbf g_n$ by~\eqref{eq:gd}, we have
\begin{equation}
    \frac L2\Vert\mathbf x_{n+1}-\mathbf x_\star\Vert^2
    %=\frac L2\Vert\mathbf x_{n}-\frac{\alpha}{L}\mathbf g_n-\mathbf x_\star\Vert^2
    =\frac L2\Vert\mathbf x_{n}-\mathbf x_\star\Vert^2-\alpha\langle\mathbf g_n,\mathbf x_n-\mathbf x_\star\rangle+\frac{\alpha^2}{2L}\Vert\mathbf g_n\Vert^2.
\end{equation}
Consequently,~\eqref{eq:CTP5} indicates \[\frac L2\Vert\mathbf x_0 - \mathbf x_\star\Vert^2+\langle u_c,v_c\rangle-(\mathbf 1^Tu_c)(f_m-f_\star)\geqslant \frac L2\Vert\mathbf x_m-\mathbf x_\star\Vert^2+\frac{C_c}{2L}\Vert\mathbf g_m\Vert^2,
\]
which implies $h_c$ and $u_c$ dominate $\frac L2\Vert\mathbf x_m-\mathbf x_\star\Vert^2+\frac{C_c}{2L}\Vert\mathbf g_m\Vert^2$. This means the SS $h_c$ is primitive by~\cref{def:primitive}, and the proof is complete.
\end{proof}

\subsection{Proof of~\cref{them-CPD}}
The following lemma is helpful.
\begin{lemma}\label{lemma:CPDtight}
Let $h_a,h_d,h_e$ and $\beta$ be the same {as in}~\cref{them-CPD}.
%Suppose $h_a\in\mathbb R_+^n$, $h_d\in\mathbb R_+^{m-n}$, and $u_d\in\mathbb R_+^{m-n+1}$,  where $m\ge n$, and $\beta = \psi(\mathbf 1^Th_a,\mathbf 1^Th_d)$ where $\psi(\cdot,\cdot)$ is defined in~\eqref{eq:def-psi}.
Further suppose {$u_d\in\mathbb R_+^{m-n}$} and $\mathbf 1^Tu_d =2(\mathbf 1^Th_d)+1$. Let  $u_e =[h_a^T,\gamma,(\lambda_1-\lambda_2)u_d^T]^T$, where $\gamma,\lambda_1,\lambda_2$ are given by
\begin{equation}\label{eq:CPD-params}
%\left\{\begin{aligned}
 \gamma = \mathbf 1^Th_a+2,\quad
\lambda_2 = \frac{2(\mathbf{1}^Th_a)+2}{\mathbf{1}^Tu_d},\quad
 \lambda_1 = \lambda_2+\frac{1}{2} + \sqrt{\lambda_2+\frac {1}{4}}  .
%\end{aligned}  \right.
\end{equation}
Then the following identities hold:
\begin{subequations}\label{eq:lemma:CPDtight}
\begin{align}
& \lambda_1 = (\lambda_1-\lambda_2)^2,\label{eq:lemma:CPDtighta}
\\ & \beta =\frac{\gamma (\lambda_1 - \lambda_2)+\lambda_2}{\lambda_1},\label{eq:lemma:CPDtightb}
\\ &\mathbf 1^Tu_e = 2(\mathbf 1^Th_e)+1.\label{eq:lemma:CPDtightc}
\end{align}
\end{subequations}
\end{lemma}

\begin{proof}%[Proof of~\cref{lemma:CPDtight}]
    Since $\lambda_1 =\lambda_2+1/2+\sqrt{\lambda_2+1/4}$ in~\eqref{eq:CPD-params}, it leads to
\begin{equation}\label{eq:CPD-l1l2}
    (\lambda_1-\lambda_2)^2=\left(\frac12\right)^2+\left(\lambda_2+\frac14\right)+\sqrt{\lambda_2+\frac14}=\lambda_1,
\end{equation}
i.e.,~\eqref{eq:lemma:CPDtighta} holds.

Let $x = \mathbf 1^Th_a$ and $y = \mathbf 1^Th_d$ for notational simplicity. Then $\mathbf 1^Tu_d = 2y+1$, $\lambda_2 = 2(x+1)/(2y+1)$ and $\gamma = x+2$. Thus, we have
\begin{equation}\label{eq:CPD-tight1}\begin{aligned}
\frac{\gamma(\lambda_1-\lambda_2)+\lambda_2}{\lambda_1}
    &=1+\frac{(\gamma-1)(\lambda_1-\lambda_2)}{\lambda_1}
 \stackrel{\eqref{eq:CPD-l1l2}}{=} 1+\frac{\gamma-1}{\lambda_1-\lambda_2}
= 1+\frac{x+1}{\lambda_1-\lambda_2}
\\ &=1+\frac{\frac{2y+1}{2}\lambda_2}{\lambda_1-\lambda_2}
=1+\frac{2y+1}{2}\left(\frac{\lambda_1}{\lambda_1-\lambda_2}-1\right)
\\ &\!\!\!\!\stackrel{\eqref{eq:CPD-l1l2}}{=}1+\frac{2y+1}{2}(\lambda_1-\lambda_2-1).
\end{aligned}\end{equation}
Substituting  $\lambda_1-\lambda_2 = 1/2+\sqrt{\lambda_2+1/4}$ into~\eqref{eq:CPD-tight1}, we obtain that
\begin{equation}\label{eq:CPD-def-of-beta}\begin{aligned} &\quad \ \frac{\gamma(\lambda_1-\lambda_2)+\lambda_2}{\lambda_1}  = 1+\frac{2y+1}{2}\cdot \left(-\frac12+\sqrt{\lambda_2+\frac14}\right)
\\ &
= 1+\frac{2y+1}{2}\cdot \left(-\frac12+\sqrt{\frac{2(x+1)}{2y+1}+\frac14}\right)
% \\ &
= \frac{3-2y+\sqrt{(2y+1)(2y+8x+9)}}{4}.
\end{aligned}
\end{equation}
\cref{eq:CPD-def-of-beta} is exactly  $\psi(x,y)$ by~\eqref{eq:def-psi}, so~\eqref{eq:lemma:CPDtightb} holds.

{Recall $u_e =[h_a^T,\gamma,(\lambda_1-\lambda_2)u_d^T]^T$.} Noting $\gamma = \mathbf 1^Th_a+2$ in~\eqref{eq:CPD-params}, we have
\begin{equation}\label{eq:CPD-1u2h}
\begin{aligned}
    \mathbf 1^Tu_e & = \mathbf 1^Th_a+\gamma + (\lambda_1 - \lambda_2)(\mathbf 1^Tu_d)
=2x+2+(\lambda_1 - \lambda_2)(2y+1).
\end{aligned}
\end{equation}
Meanwhile,  as $h_e = [h_a^T,\beta,h_d^T]^T$, we have
\begin{equation}\label{eq:CPD-1u2h2}
\begin{aligned}
    2(\mathbf 1^Th_e)+1 &~~~~=~~~~~ 2(\mathbf 1^Th_a)+2\beta+2(\mathbf 1^Th_d) +1 = 2x+2\beta + 2y+1
    \\ &\!\!\!\stackrel{(\text{\ref{eq:lemma:CPDtightb},\ref{eq:CPD-tight1}})}{=}2x+2\cdot\left(1+\frac{2y+1}{2}(\lambda_1-\lambda_2-1)\right)+2y+1
    \\ &~~~~=~~~~~(2x+2)+(\lambda_1-\lambda_2)(2y+1).
\end{aligned}
\end{equation}
Equations~\eqref{eq:CPD-1u2h} and~\eqref{eq:CPD-1u2h2} give~\eqref{eq:lemma:CPDtightc} immediately.
\end{proof}

Now we are ready to prove~\cref{them-CPD}.
\begin{proof}[Proof of~\cref{them-CPD}]
Since $h_e$ starts with the primitive SS $h_a\in\mathbb R_+^n$, we have
\begin{equation}\label{eq:CPD1}
\frac L2\Vert\mathbf x_0-\mathbf x_\star\Vert^2+\langle u_a,v_a\rangle - (\mathbf 1^Tu_a)(f_n-f_\star)\geqslant \frac L2\Vert\mathbf x_n-\mathbf x_\star\Vert^2+\frac{C}{2L}\Vert \mathbf g_n\Vert^2 ,
% \text{ where }C=(\mathbf 1^Th_a)\cdot (\mathbf 1^Th_a+1). % too long to display
\end{equation}
where $C=(\mathbf 1^Th_a)\cdot (\mathbf 1^Th_a+1)$, $u_a =[h_a^T,0]^T$ and $v_a = [Q_{0,\star},\dotsc,Q_{n,\star}]^T$.
Since $h_d\in\mathbb R_+^{m-n-1}$ is dominant and starts from the $(n+1)$-th step, by~\cref{def:dominant} there exists a vector $u_d\in\mathbb R_+^{m-n}$ such that $\mathbf 1^Tu_d =2(\mathbf 1^Th_d)+1$ and
\begin{equation}\label{eq:CPD2}
\frac L2\Vert\mathbf x_{n+1}-\mathbf x_\star\Vert^2+\langle u_d,v_d\rangle - (\mathbf 1^Tu_d)(f_m-f_\star) \geqslant \frac L2\Vert \mathbf x_{n+1} -\mathbf x_\star - L^{-1} G_du_d\Vert^2,
\end{equation}
where $v_d = [Q_{n+1,\star},\dotsc,Q_{m,\star}]^T$ and $G_d =[\mathbf g_{n+1},\dotsc,\mathbf g_m]$.

By~\cref{lemma:Qi*-ineq}, we have  $Q_{i,\star}\leqslant ( f_n -f_\star-\frac{1}{2L}\Vert \mathbf g_n\Vert^2)+\frac 1L\langle\mathbf g_i,\mathbf g_n -L(\mathbf x_n-\mathbf x_\star)\rangle$. Because $\langle u_d,v_d\rangle=\langle u_d,[Q_{n+1,\star},\dotsc,Q_{m,\star}]^T\rangle$ and $u_d\geqslant 0$, we obtain
\begin{equation}\label{eq:CPD3}
     \langle u_d, v_d\rangle \leqslant (\mathbf 1^Tu_d)(f_n-f_\star -\frac{1}{2L}\Vert \mathbf g_n\Vert^2)+\frac 1L\langle G_du_d, \mathbf g_n - L(\mathbf x_n-\mathbf x_\star)\rangle.
\end{equation}
Moreover, it follows from the definition of $Q_{n,\star}$ in~\eqref{Qij2} that
\begin{equation}\label{eq:CPD4}
 Q_{n,\star} - (f_m-f_\star)  =  f_{n}-f_m-\langle \mathbf g_n,\mathbf x_n-\mathbf x_\star\rangle +\frac{1}{2L}\Vert\mathbf  g_n\Vert^2 .
\end{equation}
Let $\gamma\geqslant 0$ and $\lambda_1\geqslant \lambda_2\geqslant 0$ be  scalars defined in~\eqref{eq:CPD-params}, we rewrite (\ref{eq:CPD1}--\ref{eq:CPD4}) as
\begin{equation}\label{eq:CPD5}\footnotesize
    \left\{\begin{aligned}
   & \frac L2\Vert\mathbf x_0-\mathbf x_\star\Vert^2+\langle u_a,v_a\rangle - (\mathbf 1^Tu_a)(f_m-f_\star)\\ &\qquad\qquad\qquad\qquad \geqslant \frac L2\Vert\mathbf x_n-\mathbf x_\star\Vert^2+\frac{C}{2L}\Vert \mathbf g_n\Vert^2  +(\mathbf 1^Tu_a)(f_n-f_m),
    \\ & \lambda_1\langle u_d,v_d\rangle -\lambda_1(\mathbf 1^Tu_d)(f_m-f_\star)\\ &\qquad\qquad\qquad\qquad  \geqslant \frac{\lambda_1 L}{2}\Vert \mathbf x_{n+1} -\mathbf x_\star - L^{-1} G_du_d\Vert^2-\frac{\lambda_1 L}{2}\Vert\mathbf x_{n+1} -\mathbf x_\star\Vert^2,
    \\ & -\lambda_2\langle u_d, v_d\rangle+\lambda_2(\mathbf 1^Tu_d)(f_m-f_\star)\\ &\qquad\qquad\qquad\qquad \geqslant -\lambda_2 (\mathbf 1^Tu_d)(f_n -f_m-\frac{1}{2L}\Vert \mathbf g_n\Vert^2)-\frac{\lambda_2}{L}\langle G_du_d, \mathbf g_n -L( \mathbf x_n-\mathbf x_\star)\rangle,
    \\ &
  \gamma(  Q_{n,\star} -  (f_m-f_\star) )= \gamma(f_{n}-f_m-\langle \mathbf g_n,\mathbf x_n-\mathbf x_\star\rangle +\frac{1}{2L}\Vert\mathbf  g_n\Vert^2) .
    \end{aligned}\right.
\end{equation}

% observe that the left hand side (LHS) of the sum is precisely to the LHS of~\eqref{eq:CPD-target2}. Thus it suffices to show that the right hand side (RHS) of the sum is equivalent to $\frac L2\Vert\mathbf x_0-\mathbf x_\star - L^{-1}G_eu_e\Vert^2$, which is the RHS of~\eqref{eq:CPD-target2}.
% observe that the left hand side (LHS) of the sum is precisely to the LHS of~\eqref{eq:CPD-target2}. Thus it suffices to show that the right hand side (RHS) of the sum is equivalent to $\frac L2\Vert\mathbf x_0-\mathbf x_\star - L^{-1}G_eu_e\Vert^2$, which is the RHS of~\eqref{eq:CPD-target2}.

We now sum up the four inequalities in~\eqref{eq:CPD5} and simplify the left hand side (LHS) and right hand side (RHS), respectively. Adding up the LHSs of~\eqref{eq:CPD5}, we obtain
\begin{equation}\label{eq:CPDLHS1}
\begin{aligned}
    \text{LHS}=&\frac L2\Vert\mathbf x_0-\mathbf x_\star\Vert^2+
    \langle u_a,v_a\rangle+\gamma Q_{n,\star}+(\lambda_1-\lambda_2)\langle u_d,v_d\rangle
   \\ & - (\mathbf 1^Tu_a+\gamma +(\lambda_1-\lambda_2)(\mathbf 1^Tu_d))(f_m-f_\star).
    \end{aligned}
\end{equation}
Let $u_e =[h_a^T,\gamma,(\lambda_1-\lambda_2)u_d^T]^T\in\mathbb R_+^{m+1}$ and $v_e = [Q_{0,\star},\dotsc,Q_{m,\star}]^T$. Then~\eqref{eq:CPDLHS1} implies
\begin{equation}\label{eq:CPDLHS2}
\text{LHS} = \frac L2\Vert\mathbf x_0 - \mathbf x_\star\Vert^2 + \langle u_e,v_e\rangle -(\mathbf 1^Tu_e)(f_m-f_\star).
\end{equation}
For the sum of RHSs of~\eqref{eq:CPD5}, we have
\begin{equation}\label{eq:CPDRHS1}\small
\begin{aligned}
   \text{RHS}&= \frac{L}{2}\Vert  \mathbf x_n-\mathbf x_\star\Vert^2+{\frac{\overbrace{C+\lambda_2(\mathbf 1^Tu_d)+\gamma}^{A_2}}{2L}}\Vert \mathbf g_n\Vert^2+(f_n-f_m)\overbrace{(\mathbf 1^Tu_a - \lambda_2(\mathbf 1^Tu_d)+\gamma)}^{A_1}
   \\ & +\frac{\lambda_1}{2L}\Vert G_du_d\Vert^2
   -\frac 1L\langle G_du_d,\overbrace{\lambda_1 L (\mathbf x_{n+1} -\mathbf x_\star)+\lambda_2\mathbf g_n-\lambda_2L(\mathbf x_n-\mathbf x_\star)}^{A_3}\rangle -\gamma \langle \mathbf g_n,  \mathbf x_n-\mathbf x_\star\rangle.
    \end{aligned}
\end{equation}
Now we simplify $A_1,A_2$ and $A_3$ in~\eqref{eq:CPDRHS1} with~\cref{lemma:CPDtight}. % by plugging in the parameters  in~\eqref{eq:CPD-params}.
%Recall   $\gamma = \mathbf 1^Th_a+2$ and $\lambda_2 = 2(\mathbf 1^Th_a+1)/(\mathbf 1^Tu_d)$ by definition.
Using $u_a=[h_a^T,0]^T$ and~\eqref{eq:CPD-params},  we have
\begin{equation}\label{eq:CPD6}A_1 =
    \mathbf 1^Th_a - \lambda_2(\mathbf 1^Tu_d)+(\mathbf 1^Th_a+2) =0.
\end{equation}
Due to $C=(\mathbf 1^Th_a)\cdot (\mathbf 1^Th_a+1)$ in~\eqref{eq:CPD1}, with~\eqref{eq:CPD-params} we have
\begin{equation}\label{eq:CPD7}\begin{aligned}
      A_2&= {C+2(\mathbf 1^Th_a+1)+\gamma}\\ & =(\mathbf 1^Th_a)\cdot (\mathbf 1^Th_a+1)+2(\mathbf 1^Th_a+1)+(\mathbf 1^Th_a+2)
       =(\mathbf 1^Th_a+2)^2
  = {\gamma^2}.
  \end{aligned}
\end{equation}
Since $\beta$ is the scaled stepsize from $\mathbf x_n$ to $\mathbf x_{n+1}$, we have the relation that $\mathbf x_{n+1}=\mathbf x_n-\frac{\beta}{L}\mathbf g_n $. Substituting $\beta=(\gamma(\lambda_1-\lambda_2)+\lambda_2)/\lambda_1$  from~\cref{lemma:CPDtight},  we get
\begin{equation}\label{eq:CPD9}\begin{aligned}A_3
%& =\lambda_1 L (\mathbf x_{n+1} -\mathbf x_\star)+\lambda_2\mathbf g_n-\lambda_2L(\mathbf x_n-\mathbf x_\star)\\
&= \lambda_1L(\mathbf x_n-\mathbf x_\star)-(\gamma(\lambda_1-\lambda_2)+\lambda_2)\mathbf g_n +\lambda_2 \mathbf g_n-\lambda_2 L(\mathbf x_n-\mathbf x_\star)
\\ &= (\lambda_1 -\lambda_2)(L(\mathbf x_n-\mathbf x_\star) -\gamma \mathbf g_n).
\end{aligned}
\end{equation}
Plugging (\ref{eq:CPD6}--\ref{eq:CPD9}) back into~\eqref{eq:CPDRHS1} and using the identity~\eqref{eq:lemma:CPDtighta}, we obtain
%the sum of right hand sides of~\eqref{eq:CPD5} simplify to
\begin{equation}\label{eq:CPD10}
   \begin{aligned}
   \text{RHS} &=\frac{L}{2}\Vert  \mathbf x_n-\mathbf x_\star\Vert^2+\frac{\gamma^2}{2L}\Vert \mathbf g_n\Vert^2+%\overbrace{\frac{(\lambda_1 - \lambda_2)^2}{2}}^{\eqref{eq:CPD8}}\Vert Gu_d\Vert^2
\frac{(\lambda_1 - \lambda_2)^2}{2L}\Vert G_du_d\Vert^2
\\ &\quad\quad\
 -\frac1L\langle G_du_d,  (\lambda_1-\lambda_2)(L(\mathbf x_n-\mathbf x_\star)-\gamma\mathbf g_n)\rangle - \gamma \langle \mathbf g_n,  \mathbf x_n-\mathbf x_\star\rangle
 \\ & =\frac L2\Vert  \mathbf x_n-\mathbf x_\star-\frac{\gamma}{L}\mathbf g_n -\frac{\lambda_1 - \lambda_2}{L}G_du_d\Vert^2.
    \end{aligned}
\end{equation}
Because $u_e = [h_a^T,\gamma, (\lambda_1-\lambda_2)u_d^T]^T$, letting $G_e =[\mathbf g_0,\dotsc,\mathbf g_m]$, we have
\begin{equation}\label{eq:CPD11}
\mathbf x_0-\mathbf x_\star -L^{-1} G_eu_e =\mathbf  x_0 -\mathbf x_\star -\sum_{i=0}^{n-1} \frac{h_i}{L}\mathbf g_i- \frac{\gamma}{L}\mathbf g_n -\frac{\lambda_1-\lambda_2}{L}G_du_d.
\end{equation}
Noting that $\mathbf x_0 - \sum_{i=0}^{n-1} \frac{h_i}{L}\mathbf g_i =\mathbf x_n$ by the definition of gradient descent method~\eqref{eq:gd},~\eqref{eq:CPD11} indicates $\mathbf x_0 -\mathbf x_\star- L^{-1}G_eu_e = \mathbf x_n-\mathbf x_\star-(\gamma/L)\mathbf g_n - (\lambda_1 - \lambda_2)L^{-1}G_du_d$. Thus~\eqref{eq:CPD10} implies that $\text{RHS}=\frac L2\Vert \mathbf x_0-\mathbf x_\star - L^{-1}G_eu_e\Vert^2$. This,  together with~\eqref{eq:CPDLHS2}, shows that  $h_e$ and $u_e$ dominate $\frac L2\Vert \mathbf x_0-\mathbf x_\star - L^{-1}G_eu_e\Vert^2$. Finally, since $\mathbf 1^Tu_e = 2(\mathbf 1^Th_e)+1$ by~\eqref{eq:lemma:CPDtightc} in~\cref{lemma:CPDtight}, we conclude that $h_e$ is a dominant SS as desired.
\end{proof}

\section{SSs from Dynamic Programming}\label{sect:applications}
In this section, we apply the concatenation techniques in \CC~to establish a set of primitive SSs $H_\algA$ and a set of dominant SSs $H_\algB$. We further show that SSs in $H_\algB$ achieves asymptotically that $f_n-f_\star\leqslant (0.4239+o(1))\cdot n^{-\varrho}\cdot \frac L2\Vert\mathbf x_0-\mathbf x_\star\Vert^2$.
% which strictly outperforms existing results.

\subsection{Algorithm to Generate Efficient SSs}\label{sect:alg_generate}
In this subsection, we propose an algorithm {for generating efficient SSs}, by recursively applying \texttt{ConPP} an \texttt{ConPD}. The algorithm is inspired by the principles of dynamic programming and consists of two phases, each recursively invoking \texttt{ConPP} or \texttt{ConPD} procedures to construct new primitive or dominant SSs from {two SSs with fewer steps}.

\begin{definition}\label{def-alg1} Let $H_\algA=\{h_\algA^{(i)}:\ i\in\N_0\}$ be  a family of primitive SSs defined recursively as follows:  $h_\algA^{(0)}=[~]^T\in\mathbb R^0$,  and  for $n\geqslant 1$,
\begin{equation}\label{eq:def-alg1}
\begin{aligned}
\mathcal H_\algA^{(n)}&=\left\{\Call{ConPP}{h_\algA^{(k)}, h_\algA^{(n-k-1)}}: \ k=0,1,\dotsc,n-1 \right\}\subset \mathbb R^n,
\\
    h_\algA^{(n)}&\in\argmax_{h}\left\{\mathbf 1^Th:\ h\in\mathcal H_\algA^{(n)}\right\},%\subset\mathbb R^n.
\end{aligned}
\end{equation}
where any maximizer may be selected if it is not unique.
%{where the choice of $h_\algA^{(n)}$ is arbitrary if the maximizer is not unique.}
\end{definition}
% \todo{Is this procedure produce best primitive SS in some sense?}

The rationale behind employing ``max'' operators in~\eqref{eq:def-alg1}  is straightforward. For each iteration count $n\geqslant 1$, the set $\mathcal{H}_\algA^{(n)}$ contains all primitive SSs of $n$ iterations that can be created by concatenating $h_\algA^{(k)}$ and $h_\algA^{(n-k-1)}$ via \texttt{ConPP}. Amongst these, the SS $h$ that maximizes $\mathbf{1}^T h$ is selected as $h_\algA^{(n)}$. 

{Importantly, this greedy selection is without loss of optimality within the class of all primitive schedules constructible from the empty schedule \( [~]^T \) by repeated application of \texttt{ConPP}. Indeed, by~\cref{them-CPP}, the total sum of any concatenated schedule \( h_c = \texttt{ConPP}(h_a, h_b) \) satisfies \(
\mathbf{1}^T h_c = \mathbf{1}^T h_a + \varphi(\mathbf{1}^T h_a, \mathbf{1}^T h_b) + \mathbf{1}^T h_b\). Since $(x+\varphi(x,y)+y)$ is increasing in both $x$ and $y$, it follows that $\mathbf{1}^T h_c$ is strictly increasing in both $\mathbf{1}^Th_a$ and $\mathbf{1}^Th_b$. Consequently, for any fixed split \( k \), the value \( \mathbf{1}^T h_c \) is maximized if and only if both \( \mathbf{1}^T h_a \) and \( \mathbf{1}^T h_b \) are individually maximized. This monotonicity property ensures that retaining only the schedule with maximal \( \mathbf{1}^T h \) at each length \( k \) suffices to recover the globally optimal schedule among all those generated by the \texttt{ConPP} recursion.} 

According to~\cref{coro:bound-primitive}, larger values of $\mathbf 1^Th$ yield better worst-case performance for primitive stepsize schedules. This is illustrated in the next proposition, whose proof follows directly from~\CPP\ and~\cref{coro:bound-primitive}.
\begin{proposition}\label{them-alg1} Any SS $h_\algA^{(n)}\in \mathbb R^n$ defined in Definition \ref{def-alg1} is primitive. Gradient descent~\eqref{eq:gd} with SS $h_\algA^{(n)}$ has
\[f_n-f_\star\leq \frac{1}{2(\mathbf 1^Th_\algA^{(n)})+1}\cdot \frac L2\Vert\mathbf x_0-\mathbf x_\star\Vert^2,\]
where equality is attained by the Huber function~\eqref{eq:huber1}.
\end{proposition}

A parallel logic applies to \texttt{ConPD} techniques to generate dominant SSs.%, where $h$ with the largest $\mathbf 1^Th$ is preserved  to optimize the worst-case complexity bound.

\begin{definition}\label{def-alg2} Let $H_\algB=\{h_\algB^{(i)}:\ i\in\N_0\}$ be  a family of dominant SSs defined recursively as follows: $h_\algB^{(0)}=[~]^T\in\mathbb R^0$,  and for $n\geqslant 1$,
\begin{equation}\label{eq:def-alg2}
\begin{aligned}
\mathcal H_\algB^{(n)}&=\left\{\Call{ConPD}{h_\algA^{(k)}, h_\algB^{(n-k-1)}}: \ k=0,1,\dotsc,n-1 \right\}\subset \mathbb R^n,
\\
    h_\algB^{(n)}&\in\argmax_{h}\left\{\mathbf 1^Th:\ h\in\mathcal H_\algB^{(n)}\right\},%\subset\mathbb R^n,
\end{aligned}
\end{equation}
where $h_\algA^{(k)}$ is the primitive SS defined in Definition \ref{def-alg1},
and any maximizer may be selected if it is not unique.
%{the choice of $h_\algB^{(n)}$ is arbitrary if the maximizer is not unique.}
\end{definition}

The complexity bound of the dominant SSs generated by the procedure in~\cref{def-alg2} follows from \CPD\ and~\cref{them-dominate}, as illustrated below.
\begin{proposition}\label{them-alg2} Any SS $h_\algB^{(n)}\in \mathbb R^n$ defined in Definition \ref{def-alg2} is dominant.
Gradient descent~\eqref{eq:gd} with SS $h_\algB^{(n)}$ has
\[f_n-f_\star\leq \frac{1}{2(\mathbf 1^Th_\algB^{(n)})+1}\cdot \frac L2\Vert\mathbf x_0-\mathbf x_\star\Vert^2, \]
where equality is attained by the Huber function~\eqref{eq:huber1}.
\end{proposition}

In the following two propositions, we show that the key quantities $\mathbf 1^Th_\algA^{(n)}$ and $\mathbf 1^Th_\algB^{(n)}$ can be computed independently, without demanding the full vectors $h_\algA^{(n)}$ or $h_\algB^{(n)}$.

\begin{proposition}\label{prop-recur1}
The sequence $\{\mathbf 1^Th_\algA^{(n)}\}$ satisfies the following recursion:
\begin{equation}\label{prop-eq-recur-r1}\begin{aligned}
    \mathbf 1^Th_\algA^{(n)}&=\max_{0\leqslant k<n}\left\{\mathbf 1^Th_\algA^{(k)} + \varphi(\mathbf 1^Th_\algA^{(k)},\mathbf 1^Th_\algA^{(n-k-1)})+\mathbf 1^Th_\algA^{(n-k-1)}    \right\},
    \end{aligned}
\end{equation}
where $\varphi$ is defined in~\eqref{eq:def-varphi} and $n\geqslant 1$.
%where $\varphi(x,y)= (-x-y+\sqrt{(x+y+2)^2+4(x+1)(y+1)} )/2$ and $n\geqslant 1$.
\end{proposition}

\begin{proposition}\label{prop-recur2}
The sequence $\{\mathbf 1^Th_\algB^{(n)}\}$ satisfies the following recursion:
\begin{equation}\label{prop-eq-recur-r2}\begin{aligned}
\mathbf 1^Th_\algB^{(n)}=\max_{0\leqslant k<n}\left\{\mathbf 1^Th_\algA^{(k)}+\psi(\mathbf 1^Th_\algA^{(k)},\mathbf 1^Th_\algB^{(n-k-1)})+\mathbf 1^Th_\algB^{(n-k-1)}\right\},
    \end{aligned}
\end{equation}
where $\psi$ is defined in~\eqref{eq:def-psi} and $n\geqslant 1$.
%where $\psi(x,y) = (3-2y+\sqrt{(2y+1)(2y+8x+9)})/4$ and $n\geqslant 1$.
\end{proposition}

Now let us introduce an algorithm that produces an SS for any given $n$.
We first generate a collection of primitive SSs $H_\algA$ using~\cref{def-alg1}, which is summarized in~\cref{alg1}.
Then we use this collection to generate a new set of dominant SSs $H_\algB$ using~\cref{def-alg2}, which is summarized in~\cref{alg2}.
{ As demonstrated in Propositions \ref{prop-recur1} and \ref{prop-recur2}, $\mathbf 1^Th_\algA^{(n)}$ and $\mathbf 1^Th_\algB^{(n)}$ can be calculated efficiently by storing all the previous values of  $\mathbf 1^Th_\algA^{(k)}$ and $\mathbf 1^Th_\algB^{(k)}$ for $k<n$. This enables the computation of stepsize schedules for all $n=1,\dotsc,N$ in $O(N^2)$ time. The time and space complexities are summarized in~\cref{them-time-storage}. A Python implementation of the algorithms is available at~\colorurl{https://github.com/ForeverHaibara/concat-stepsize/tree/main}.}

\begin{center}\cbl
\begin{minipage}{0.49\textwidth}
\begin{algorithm} [H]\small
\caption{\textproc{priDP}($N$)} \label{alg1}
\begin{algorithmic}[1]\cbl
\State   $H_\algA[0] \gets  [~]^T$
\State $S_\algA[0]\gets 0$
\For{$n \gets 1$ to $N$}
    \State $k_{\text{best}}\gets 0$, $s_{\text{max}} \gets 0$
    \For{$k \gets 0$ to $n-1$}
        \State $k'\gets n-k-1$
        \State $s\gets S_\algA [k]+\varphi(S_\algA[k],S_\algA[k'])+S_\algA[k']$
        \If{$s \geqslant s_{\text{max}}$}
            \State $k_{\text{best}} \gets k$, $s_{\text{max}} \gets s$
        \EndIf
    \EndFor
    \State $k\gets k_{\text{best}}$, $k'\gets n-k_{\text{best}}-1$
    \State $H_\algA[n] \gets  \Call{ConPP}{H_\algA[k], H_\algA[k']}$
    \State $S_\algA[n]\gets s_{\text{max}}$
\EndFor
\State \textbf{return} $H_\algA$, $S_\algA$
%\EndFunction
\end{algorithmic}
\end{algorithm}
\end{minipage}
\hfill
\begin{minipage}{0.499\textwidth}
\begin{algorithm}[H] \small
\caption{{\textproc{domDP}($N$)}}    \label{alg2}
\begin{algorithmic}[1]\cbl
\State   $H_\algB[0] \gets [~]^T$, $S_\algB[0]\gets 0$ \State $H_\algA,S_\algA\gets\Call{priDP}{N}$ \label{line:invoke-alg1}
\For{$n \gets 1$ to $N$}
    \State $k_{\text{best}}\gets 0$, $s_{\text{max}} \gets 0$
    \For{$k \gets 0$ to $n-1$}        
        \State $k'\gets n-k-1$
        \State $s\gets S_\algA[k]+\psi(S_\algA[k],S_\algB[k'])+S_\algB[k']$
        \If{$s \geqslant s_{\text{max}}$}
            \State $k_{\text{best}} \gets k$, $s_{\text{max}} \gets s$
        \EndIf
    \EndFor 
    \State $k\gets k_{\text{best}}$, $k'\gets n-k_{\text{best}}-1$
    \State $H_\algB[n] \gets  \Call{ConPD}{H_\algA[k], H_\algB[k']}$
    \State $S_\algB[n]\gets s_{\text{max}}$
\EndFor
\State \textbf{return} $H_\algB$, $S_\algB$
% \EndFunction
\end{algorithmic}
\end{algorithm}
\end{minipage}
\end{center}

\begin{theorem}\label{them-time-storage} Given an integer $N\ge1$, the values of $h_\algA^{(n)}\in\mathbb R^n$ in Definition \ref{def-alg1} and $h_\algB^{(n)}\in\mathbb R^n$ in Definition \ref{def-alg2} for $n=0,1,\dotsc,N$ can be computed using Algorithms \ref{alg1} and \ref{alg2}, respectively. Both algorithms run in $O(N^2)$ time and require $O(N^2)$ storage.
\end{theorem}
\begin{proof} 
We analyze~\cref{alg1}, and the analysis for~\cref{alg2} is analogous. The algorithm maintains two arrays: \(H_\algA[n]\) storing the actual stepsize vector of length \(n\), and \(S_\algA[n] = \mathbf{1}^T H_\algA[n]\) storing its sum. For each \(n = 1, \dotsc, N\), the inner loop evaluates all \(n\) candidate splits (for \(k = 0, \dotsc, n-1\)) using only scalar arithmetic via the recurrence in~\cref{prop-recur1}, which takes \(O(n)\) time for each $n$. After identifying the optimal index \(k_{\text{best}}\), the algorithm constructs \(H_\algA[n]\) by calling \texttt{ConPP}. This operation creates a vector of length \(n\) and therefore costs \(O(n)\) time. Summing over \(n = 1\) to \(N\), the total time complexity is \(\sum_{n=1}^N O(n) = O(N^2)\). The total space complexity is also $O(N^2)$, dominated by the storage of all vectors in $\{H_\algA[n]\}_{n=0}^N$. The same reasoning applies to~\cref{alg2}, which incurs an additional $O(N^2)$ time and space after invoking~\cref{alg1}.
\end{proof}

% \begin{proof} For simplicity, we only demonstrate  the results for $h_\algA$. We store the values of $\mathbf 1^Th_\algA^{(n)}$ for $n=0,1,\ldots,N$.
% At each iteration $n$, we demand the optimal index $k$ such that the concatenated $h$ has maximum value $\mathbf 1^Th$, as demonstrated in lines \ref{line:6}--\ref{line:8} of~\cref{alg1}. 
% Instead of explicitly simulating all possible concatenations as line \ref{line:5} does, the optimal index $k$  can be efficiently determined by simple arithmetic using~\cref{prop-recur1}. After $k$ is determined, we perform the actual concatenation $h_\algA^{(n)} =\Call{ConPP}{h_\algA^{(k)}, h_\algA^{(n-k-1)}}$, requiring only a single vector concatenation operation.
% % the dominant cost is in line \ref{line:6} of~\cref{alg1}. \todo{add a sentence}
% % We first determine the optimal index $k$ using~\cref{prop-recur1}, and then perform the actual concatenation $h_\algA^{(n)} =\Call{ConPP}{h_\algA^{(k)}, h_\algA^{(n-k-1)}} $.
% This process requires $O(n)$ time and   $O(n)$ storage, both of which accumulate to $O(N^2)$ as $n$ goes from $1$ to $N$.  The computation of $h_\algB$ incurs an additional  $O(N^2)$ time and storage after invoking~\cref{alg1} in line \ref{line:invoke-alg1} of~\cref{alg2}.
% \end{proof}

We remark that the values $\mathbf 1^Th_\algA^{(n)}$ and $\mathbf 1^Th_\algB^{(n)}$ for $n=0,1,\dotsc,N$ can be computed within $O(N)$ storage if the explicit SSs are not required.

%It is also worth noting that there might be multiple maximizers $h$ in~\eqref{eq:def-alg1} or~\eqref{eq:def-alg2}, and any of these maximizers is acceptable. The issue of non-uniqueness will be further discussed in~\cref{sect:nonunique}.

\subsection{Asymptotic Convergence}\label{sect:asymptotic}
% In this section, we will show that our methods have a better asymptotic convergence rate than existing methods in the sense of a better constant factor. %Our results reveal that  $h_\algA^{(n)}$ and $h_\algB^{(n)}$  achieve a convergence rate of $O(n^{-\log_2(\sqrt{2}+1)}) \approx O(n^{-1.2716})$ as $n\in\N$ tends to the infinity.
% Another advantage of  our method is that the bound is for an arbitrary iteration number  $n\in\mathbb N$, while the methods in~\cite{silver2} and~\cite{grimmer2024acceleratedobjective} has a limitation that requires $n=2^l-1$, $l\in \mathbb N$, although they both have an $O(n^{-\varrho})$ asymptotic convergence rate\footnote{We remark that using a doubling trick or tracking the best obtained objective (see the last paragraph of Section 1 in~\cite{kornowski2024open}), one can show that the bound  $O(n^{-\varrho})$  holds for every $n$.  However, this approach will worsen the constant factor in the bound.}, where $\varrho = \log_2(\sqrt 2+1)$.

{In this subsection, we analyze the asymptotic behavior of the SSs constructed in~\cref{sect:alg_generate}.
For readability, we defer the proofs of Propositions~\ref{them-convergence1},~\ref{them-convergence1.2},~\ref{them-convergence2.1} and~\ref{them-convergence2.2} to \cref{appendix:asymptotic}.
}

%{delete?Recall that Propositions \ref{them-alg1} and \ref{them-alg2} establish bounds of the form $f_n-f_\star \leqslant \frac{1}{2(\mathbf{1}^T h )+ 1} \cdot \frac{L}{2} \Vert \mathbf{x}_0 -\mathbf x_\star\Vert^2$. Therefore, our analysis focuses on the growth rates of $\mathbf{1}^T h_\algA^{(n)}$ and $\mathbf{1}^T h_\algB^{(n)}$.}

\subsubsection{Convergence Rate for SS $h_\algA^{(n)}$}
We commence with $\mathbf{1}^T h_\algA^{(n)}$, presenting Propositions \ref{them-convergence1} and \ref{them-convergence1.2} to demonstrate that $(\sqrt{2} - 1)(n+2)^\varrho \leqslant \mathbf{1}^T h_\algA^{(n)} + 1 \leqslant (n+1)^\varrho$ holds universally for $n \in \mathbb{N}_0$.

%%%%%%%%%%%%%%%%%%%%%%%%%%%%%%%%%%%%%%
%
%
%         Bound for h_\algA
%
%
%%%%%%%%%%%%%%%%%%%%%%%%%%%%%%%%%%%%%%

\begin{proposition}[Upper Bound for $\mathbf 1^Th_\algA^{(n)}$] \label{them-convergence1} It holds that $\mathbf 1^Th_\algA^{(n)}+1\leqslant (n+1)^\varrho$. Moreover, when $n = 2^l-1$, $l\in\N_0$, it holds that $\mathbf 1^Th_\algA^{(n)}+1=(n+1)^\varrho$.
\end{proposition}

\begin{proposition}[Lower Bound for $\mathbf 1^Th_\algA^{(n)}$] \label{them-convergence1.2} Define $\nu_l ={\displaystyle {\min_{2^l-1\leqslant n\leqslant 2^{l+1}-2}}}\frac{\mathbf 1^Th_\algA^{(n)}+1}{(n+2)^\varrho}$, where $l\in\N_0$. Then the sequence $\{\nu_l\}_{l=0}^\infty$ is  non-decreasing.
Consequently, given that $\nu_0 = \sqrt 2- 1$, we have
$\mathbf 1^Th_\algA^{(n)}+1\geqslant (\sqrt 2-1)(n+2)^\varrho$ for all $n\in\mathbb N_0$.
\end{proposition}

\begin{remark}\cbl\label{them-convergence1.3} 
Using Python's standard floating-point arithmetic, we numerically observe that $\nu_{16} > 0.9928$, where $\{\nu_l\}$ is defined as in~\cref{them-convergence1.2}. This suggests that
$\liminf_{n\to\infty} \mathbf 1^Th_\algA^{(n)}/n^\varrho> 0.9928$.
We state this as a numerical observation rather than a proven result. 
\end{remark}
% Using Python's standard floating-point arithmetic, we verify that  $\nu_{16} > 0.9928$, where $\{\nu_l\}$ is defined as in~\cref{them-convergence1.2}. We then immediately have the following corollary.
% \begin{corollary}\label{them-convergence1.3} 
% %Let $\bar \nu_\algA =\liminf_{n\to\infty} \mathbf 1^Th_\algA^{(n)}/n^\varrho $. It holds that $\bar\nu_\algA > 0.9928$.    
% It holds that $\liminf_{n\to\infty} \mathbf 1^Th_\algA^{(n)}/n^\varrho> 0.9928$.  
% %  more precise (computed with \nu_{27}): 0.9928193+
% \end{corollary}

\begin{figure}[h]
    \centering
    \includegraphics[width=\textwidth]{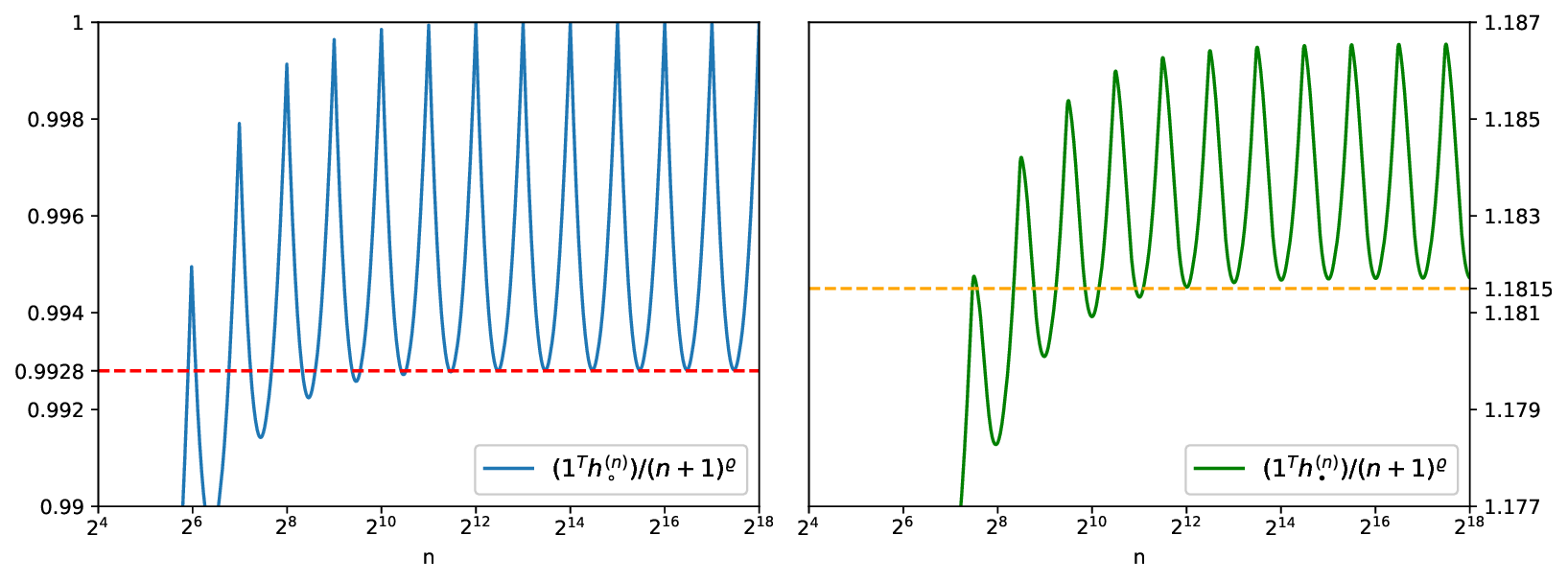}
    \vspace{-0.7cm}
    \caption{Plots of $(\mathbf 1^T h_\algA^{(n)})/(n+1)^\varrho$ (left) and $(\mathbf 1^T  h_\algB^{(n)})/(n+1)^\varrho$ (right) for $2^4\leqslant n\leqslant 2^{18}$,  and $h_\algA^{(n)}$ and $h_\algB^{(n)}$ are SSs defined in Definitions \ref{def-alg1} and \ref{def-alg2}, respectively. The x-axes are displayed in logarithmic scale.}
    \label{fig:stepsize}
\end{figure}

~\cref{them-convergence1} asserts that $\limsup_{n\to\infty} \mathbf{1}^T h_\algA^{(n)} / n^\varrho = 1$. However,  we remark that this does not imply $\mathbf{1}^T h_\algA^{(n)} / n^\varrho \to 1$ as $n \to \infty$. In~\cref{fig:stepsize} (left), we plot the numerical values of $\mathbf 1^Th_\algA^{(n)}/(n+1)^\varrho$. We observe an oscillating pattern in the ratio, suggesting potential non-convergence. Despite this,~\cref{them-convergence1.2} provides a tool to estimate the lower bound, leading to the observation that $\liminf_{n\to\infty} \mathbf{1}^T h_\algA^{(n)} / n^\varrho > 0.9928$ in~\cref{them-convergence1.3}.

% \bigskip

%%%%%%%%%%%%%%%%%%%%%%%%%%%%%%%%%%%%%%
%
%
%         Bound for h_\algB
%
%
%%%%%%%%%%%%%%%%%%%%%%%%%%%%%%%%%%%%%%

\subsubsection{Convergence Rate for SS $h_\algB^{(n)}$}
We now bound $\mathbf 1^Th_\algB^{(n)}$, based on previous derived bounds for $\mathbf 1^Th_\algA^{(n)}$.

In this subsection, we use the convention that
\begin{equation}\label{eq:omega}
\omega=\max_{\mu\in (0,1)}\frac{2(1-\mu)^\varrho}{1-\mu^{\varrho/2}}.
\end{equation}
It follows that $\omega\approx 2.376373$ and $\omega\ge 2.376373$.
\begin{proposition}[Upper Bound for $\mathbf 1^Th_\algB^{(n)}$]\label{them-convergence2.1} It holds that $2(\mathbf 1^Th_\algB^{(n)})+1\leqslant \omega(n+1)^\varrho$, where $\omega$ is defined in~\eqref{eq:omega}.
\end{proposition}

% Using~\cref{them-convergence2.1}, we have the following proposition that gives a lower bound for $\mathbf 1^Th_\algB^{(n)}$. It shows that $\mathbf{1}^T h_\algB^{(n)}$ exhibits a constant enhancement over $\mathbf{1}^T h_\algA^{(n)}$.  This result can be further used to derive an estimate of worst-case complexity of $h_\algB^{(n)}$.

\begin{proposition}[Lower Bound for $\mathbf 1^Th_\algB^{(n)}$] \label{them-convergence2.2} Let $\bar\nu_\algA := \liminf_{n\to\infty}  \mathbf 1^Th_\algA^{(n)}/ n^\varrho $ and $\bar\nu_\algB := \liminf_{n\to\infty}  \mathbf 1^Th_\algB^{(n)}/ n^\varrho $. Then $\bar\nu_\algB\geqslant \frac12\omega\bar\nu_\algA $, where $\omega$ is defined in~\eqref{eq:omega}.
\end{proposition}

%Let $\mu\approx 0.3187$ be the  maximizer of $\max_{\mu\in(0,1)}\frac{2(1-\mu)^\varrho}{1-\mu^{\varrho/2}}$, leading to $C\geqslant \omega \bar\nu_\algA$ as desired.

Now we give an estimate of the objective bound for SS $h_\algB^{(n)}$, which exhibits state-of-the-art convergence rate.
\begin{theorem}[Objective bound for $h_\algB^{(n)}$] 
\label{them-convergence2.29} 
Let $\bar\nu_\algB:=\liminf_{n\to\infty}\mathbf 1^Th_\algB^{(n)}/n^\varrho$.  Gradient descent~\eqref{eq:gd} with $h_\algB^{(n)}$ enjoys an asymptotic bound of
 \begin{equation}\label{eq:asym-bound1}
     f_n-f_\star\leqslant ((2\bar\nu_\algB)^{-1}+o(1))\cdot n^{-\varrho}\cdot \frac L2\Vert\mathbf x_0-\mathbf x_\star\Vert^2,
 \end{equation} 
where $\varrho =\log_2(\sqrt 2+1)$.
% and  $(2\bar\nu_\algB)^{-1}<0.4239$.
\end{theorem}
\begin{proof}This is a direct corollary of~\cref{them-alg2}.
%[Proof of~\cref{them-convergence2.3}]
% By~\cref{them-convergence1.3} and~\cref{them-convergence2.2}, we get 
% $\bar\nu_{\algB}\geqslant 0.9928\times 2.376373/2 > 1.1796$. Plugging it into~\cref{them-alg2} yields the bound~\eqref{eq:asym-bound1}.
\end{proof}

\begin{remark}\label{them-convergence2.3}
    Based on the numerical observation in~\cref{them-convergence1.3} 
and combining it with~\cref{them-convergence2.2}, we estimate that
$\bar\nu_{\algB} \geqslant \frac{0.9928 \times 2.376373}{2} > 1.1796$,
and thus $(2\bar\nu_{\algB})^{-1} < 0.4239,$
% \[
% \bar\nu_{\algB} \geqslant \frac{0.9928 \times 2.376373}{2} > 1.1796,
% \quad\text{and thus}\quad (2\bar\nu_{\algB})^{-1} < 0.4239,
% \]
where $\bar\nu_{\algB}$ is defined in~\cref{them-convergence2.29}.
We emphasize that this bound is contingent upon the numerical evidence in~\cref{them-convergence1.3}.
\end{remark}

\begin{remark}\label{remark:silver}
Prior to our work, Altschuler and Parrilo~\cite{silver2} presented the bound $f_n-f_\star\leqslant (1+o(1))\cdot n^{-\varrho} \cdot \frac{L}{2} \Vert \mathbf{x}_0 -\mathbf x_\star\Vert^2$, while Grimmer et al.~\cite{grimmer2024acceleratedobjective} presented $f_n-f_\star \leqslant (0.4302 + o(1))\cdot n^{-\varrho} \cdot \frac{L}{2} \Vert \mathbf{x}_0-\mathbf x_\star \Vert^2$. Both bounds are restricted to $n = 2^l - 1$, {$l\in\N_0$}.\footnote{We remark that using a doubling trick or tracking the best obtained objective (see the last paragraph of Section 1 in~\cite{kornowski2024open}), one can show that the bound  $O(n^{-\varrho})$  holds for every $n$.  However, this approach will worsen the constant factor in the bound.}~\cref{them-convergence2.3} suggests that our SS $h_{\algB}^{(n)}$ not only improves the existing best bound  by a constant factor $\frac{0.4302}{0.4239}\approx 1.015$,  but also applies to every $n\in\mathbb N_0$.
Moreover, our results reduce to~\cite{silver2} and~\cite{grimmer2024acceleratedobjective} with certain use of our concatenation techniques.
The last paragraph of the proof of~\cref{them-convergence1} in~\cref{appendix:asymptotic} shows that $\mathbf 1^Th_\algA^{(n)}+1=(n+1)^\varrho$ for $n=2^l - 1$ is attained via an inductive concatenation  $h_\algA^{(2^l-1)}=\Call{ConPP}{h_\algA^{(2^{l-1}-1)},h_\algA^{(2^{l-1}-1)}}$, which is exactly the silver SS established in~\cite{silver2}.
%\todo{location}
Further, if we define a new recursion  $h^{(2^l-1)} = \Call{ConPD}{h_\algA^{(2^{l-1} - 1)}, h^{(2^{l-1}-1)}}$ where $h^{(0)}=[~]^T$, then it is the regime in~\cite{grimmer2024acceleratedobjective}. The proposed SS $h_\algB^{(n)}$ in~\cref{def-alg2}, having employed a maximum operator to seek for the optimal concatenation, is expected to achieve better performance.

\end{remark}

\begin{remark}
The constant $0.4239$ in~\cref{them-convergence2.3} may not be the optimal constant. As shown in~\cref{fig:stepsize} (right), numerical evidence suggests a tighter estimate of $\liminf_{n\to\infty} \mathbf{1}^T h_{\algB}^{(n)} / n^\varrho$ ($>1.1815$), thus improving the statement  in~\cref{them-convergence2.3}, according to~\cref{them-alg2}. 
{\cbl
If proven, then $0.4239$ could be replaced by $0.4232$.\footnote{A concurrent work by Grimmer et al.~\cite{grimmercomposing} in October 2024 showed by a tighter estimate that $(2\bar\nu_\algB)^{-1}\in [0.4208, 0.42312]$.}}
\end{remark}

\subsection{Non-uniqueness}\label{sect:nonunique}

Both Algorithms~\ref{alg1} and \ref{alg2} employ dynamic programming to identify the best index $k$ for concatenating  SSs of iterations $k$ and $n-k-1$, given iteration number $n$. However, it is important to note that such best index $k$ might not be unique. Since the definition of $\alpha$ in~\CPP\  is symmetric with respect to the two primitive SSs $h_a$ and $h_b$, swapping $h_a$ and $h_b$ does not alter the value of $\mathbf 1^Th_c$. 
As an example, concatenating SSs $h_\algA^{(1)}=[\sqrt 2]^T$ and $h_\algA^{(0)}=[~]^T$ yields $h_\algA^{(2)} =\Call{ConPP}{h_\algA^{(1)},h_\algA^{(0)}}= [\sqrt 2,({\sqrt{10+8\sqrt 2}-\sqrt 2})/2]^T$. Meanwhile,  concatenation in 
a reverse manner leads to  $\Call{ConPP}{h_\algA^{(0)},h_\algA^{(1)}}=[({\sqrt{10+8\sqrt 2}-\sqrt 2})/2,\sqrt 2]^T$, equivalent in the sense of  $\mathbf 1^Th$. More nontrivial examples beyond the ``swapping'' exist. For instance, one can verify that $\mathbf 1^Th_\algA^{(5)}=\mathbf 1^T\Call{ConPP}{h_\algA^{(3)}, h_\algA^{(1)}}=\mathbf 1^T\Call{ConPP}{h_\algA^{(2)}, h_\algA^{(2)}}=\mathbf 1^T\Call{ConPP}{h_\algA^{(1)}, h_\algA^{(3)}}$. Therefore, there are three different possibilities for $h_\algA^{(5)}$ in Definition \ref{def-alg1}. Based on the above simple examples and our numerical tests, we conjecture the optimal index has the following coincidence.
\begin{conjecture}\label{conj:pridp}For each $n\in\N$, it holds that
\begin{equation*}
   \mathbf 1^Th_\algA^{(n)} =  \mathbf 1^T\Call{ConPP}{h_\algA^{\lfloor (n-1)/2\rfloor}, h_\algA^{\lceil (n-1)/2\rceil}}.
\end{equation*}
Moreover, if $n = p\cdot 2^l - 1$ where $p>1$ is odd and $l\in\N$, we additionally have
\begin{equation*}
    \mathbf 1^Th_\algA^{(n)} = \mathbf 1^T\Call{ConPP}{h_\algA^{((p+i)\cdot 2^{l-1}-1)}, h_\algA^{((p-i)\cdot 2^{l-1}-1)}},~i\in\{-1,0,1\}.
    % = \mathbf 1^T\Call{ConPP}{h_\algA^{(x\cdot 2^{l-1}-1)}, h_\algA^{(x\cdot 2^{l-1}-1)}}.
\end{equation*}
\end{conjecture}

{We numerically verified~\cref{conj:pridp} for all $n\leq 10^4$, with the absolute difference between the two sides of the conjectured equality not exceeding $6\times 10^{-11}$. The script is available at~\colorurl{https://github.com/ForeverHaibara/concat-stepsize/tree/main}. A direct consequence of~\cref{conj:pridp} is that it boosts \textproc{priDP}~(\cref{alg1}) to $O(n)$ time complexity, assuming the conjecture holds. We have also conducted extensive numerical experiments to search for analogous patterns in the optimal pivot selection of \textproc{domDP}~(\cref{alg2}). However, unlike the \textproc{priDP} case, the optimal index $k$ for \textproc{domDP} does not exhibit a clear and simple structure across different values of $n$. Consequently, we are unable to formulate a similarly concise conjecture for \textproc{domDP}.}

Given the non-uniqueness of optimal SSs across iterations within $H_\algA$, SSs in $H_\algB$ are also not unique. For instance, for $n=4$ case, it can be verified that $h_\algB^{(4)}=\Call{ConPD}{h_\algA^{(2)}, h_\algB^{(1)}}$ is the optimal concatenation. However, as we have demonstrated above, there are two possible values for $h_\algA^{(2)}$ with identical $\mathbf 1^Th_\algA^{(2)}$. Thus, both  schedules
  $[ \sqrt 2, ({\sqrt{10+8\sqrt 2}-\sqrt 2})/{2},\beta,3/2]^T$ and $[  ({\sqrt{10+8\sqrt 2}-\sqrt 2})/{2},\sqrt 2,\beta,3/2]^T$ enjoy the bound $f_4-f_\star\leqslant 0.06234\cdot \frac L2\Vert\mathbf x_0-\mathbf x_\star\Vert^2$,  where $\beta\approx 3.005$.
  % $\beta = \sqrt{3+\sqrt 2+\sqrt{10+8\sqrt 2}}\approx 3.005$.

 Although multiple optimal choices for $h_\algA^{(n)}$ or $h_\algB^{(n)}$ may exist, values of $\mathbf 1^Th_\algA^{(n)}$ and $\mathbf 1^Th_\algB^{(n)}$ are unaffected by this variability, as ensured by Propositions \ref{prop-recur1} and \ref{prop-recur2}. Therefore, the convergence bounds of Propositions \ref{them-alg1} and \ref{them-alg2} do not depend on the choice of the maximizers.

\subsection{Dynamic SS for Reducing Objective Values}\label{sect:dynamic}

In Definition \ref{def-alg2}, we have constructed SSs $h_\algB^{(n)}$ with notable performance,
but they require the iteration count $n$ to be specified in advance. Teboulle and Vaisbourd~\cite{teboulle2023elementary}  introduced a dynamic stepsize sequence that evolves by incrementally appending a single step to the end of the sequence at each iteration. We will show that our \texttt{ConPP} technique generalizes the process by permitting the addition of multiple steps simultaneously, thus accelerating the convergence rate.

% The silver SS in~\cite{silver2} allows extension of the stepsize sequence on $n=2^l-1$.

\begin{proposition}[Dynamic Stepsize 1]\label{them-dynamic} Suppose $h^{(0)}\in\mathbb R_+^{n_0}$ and $h_b\in\mathbb R_+^{m}$ are both primitive SSs. Define the recursion $h^{(k)} = \Call{ConPP}{h^{(k-1)}, h_b}\in \mathbb R_+^{n_k}$ for $k\in\N$. Then $n_k= n_0+k(m+1)$ and $\lim_{k\to\infty}\mathbf 1^Th^{(k)} /n_k =2(\mathbf 1^Th_b+1)/(m+1)$. Gradient descent~\eqref{eq:gd} with stepsize sequence $h^{(k)}$  enjoys the bound
\begin{equation}\label{eq:ds1}
f_{n_k}-f_\star\leqslant \frac{1}{2(\mathbf 1^Th^{(k)})+1}\cdot \frac L2\Vert\mathbf x_0-\mathbf x_\star\Vert^2\simeq \frac{m+1}{4n_k(\mathbf 1^Th_b+1)}\cdot \frac L2\Vert\mathbf x_0-\mathbf x_\star\Vert^2,
\end{equation}
where $\simeq$ denotes equality up to lower-order terms. In particular, if $h_b=h_\algA^{(m)}$,  we have
\begin{equation}\label{eq:ds11}
f_{n_k}-f_\star\leqslant
\left(\frac{m+1}{4(\sqrt 2-1)(m+2)^{\varrho}}+o(1)\right)\cdot \frac{1}{n_k}\cdot \frac L2\Vert\mathbf x_0-\mathbf x_\star\Vert^2.
\end{equation}
\end{proposition}

\begin{proof}
    It follows from~\CPP~that $h^{(k+1)} = [(h^{(k)})^T, \alpha^{(k)}, h_b^T]^T$, where $\alpha^{(k)}$ equals to $\varphi(\mathbf 1^Th^{(k)}, \mathbf 1^Th_b)$ and $\varphi(\cdot,\cdot)$ is given by~\eqref{eq:def-varphi}.
This implies that  $h^{(k)}\in\mathbb R^{n_0+k(m+1)}$, and the recursive formula 
\begin{equation}
\label{eq:dynamic1}
\begin{aligned}
    \mathbf 1^Th^{(k+1)} &= \mathbf 1^Th^{(k)} + \alpha^{(k)} + \mathbf 1^Th_b
\\ & = \frac{\mathbf 1^Th^{(k)}+\mathbf 1^Th_b+\sqrt{(\mathbf 1^Th^{(k)}+ \mathbf 1^Th_b+2)^2+4(\mathbf 1^Th^{(k)}+1)(\mathbf 1^Th_b+1)}}{2}.
\end{aligned}
\end{equation}
Equation~\eqref{eq:dynamic1} 
% ensures that $\mathbf 1^Th^{(k+1)}> \mathbf 1^Th^{(k)}+\mathbf 1^Th_b$, so the sequence $\{\mathbf 1^Th^{(k)}\}_{k=0}^\infty$  is montonically increasing and unbounded. 
shows $\mathbf 1^Th^{(k+1)}=\mathbf 1^Th^{(k)}+2(\mathbf 1^Th_b+1)+o(1)$ as $k\to\infty$. Applying the Stolz-Ces\`{a}ro theorem, we get $ \lim_{k\to\infty}  {\mathbf 1^Th^{(k)}}/{k}=2(\mathbf 1^Th_b+1)$.
% Recalling the Stolz-Ces\`{a}ro theorem that $\lim_{k\to\infty} \frac{a_k}{k}=\lim_{k\to\infty}  \frac{a_{k+1}-a_k}{(k+1)-k}$ holds if the latter limit exists, we obtain
% \begin{equation*}%\label{eq:dynamic2}
% \begin{aligned}
% &\quad\ \lim_{k\to\infty}\frac{\mathbf 1^Th^{(k)}}{k} = \lim_{k\to\infty}\frac{\mathbf 1^Th^{(k+1)} -\mathbf 1^Th^{(k)} }{(k+1)-k}
% % \\ &\!\!\!\! \stackrel{\eqref{eq:dynamic1}}{=}
% \\ &=\lim_{\mathbf 1^Th^{(k)}\to+\infty}
% \frac{-\mathbf 1^Th^{(k)}+\mathbf 1^Th_b+\sqrt{(\mathbf 1^Th^{(k)}+ \mathbf 1^Th_b+2)^2+4(\mathbf 1^Th^{(k)}+1)(\mathbf 1^Th_b+1)}}{2}\\
% &=\lim_{\mathbf 1^Th^{(k)}\to+\infty} \frac{8(\mathbf 1^Th^{(k)}+1)(\mathbf 1^Th_b+1)}{2\left(\mathbf 1^Th^{(k)}-\mathbf 1^Th_b+\sqrt{(\mathbf 1^Th^{(k)}+ \mathbf 1^Th_b+2)^2+4(\mathbf 1^Th^{(k)}+1)(\mathbf 1^Th_b+1)}\right)}\\
% % % this line is too long
%  % &=\lim_{\mathbf 1^Th^{(k)}\to+\infty}\frac{\mathbf -1^Th^{(k)}+\mathbf 1^Th_b+(\mathbf 1^Th^{(k)}+\mathbf 1^Th_b+2+2(\mathbf 1^Th_b+1)+o(1))}{2}\\
%  &=\frac{8(\mathbf 1^Th_b+1)}{2\cdot2}
%  =2(\mathbf 1^Th_b+1).
% \end{aligned}
% \end{equation*}
This, together with  $n_k=n_0+k(m+1)$, implies $\lim_{k\to\infty}\mathbf 1^Th^{(k)} /n_k =2(\mathbf 1^Th_b+1)/(m+1)$ as desired.
Equation~\eqref{eq:ds1} then follows directly from~\cref{coro:bound-primitive}.
Equation~\eqref{eq:ds11} follows from~\eqref{eq:ds1} and~\cref{them-convergence1.2}.
\end{proof}

\begin{remark}\label{remark:teboulle}
    Let $h^{(0)}=h_b=[~]^T\in\mathbb R^0$ be the zero-iteration stepsize introduced in Example \ref{example-zero}. Then $h^{(k)}\in\mathbb R^{k}$ in~\cref{them-dynamic} is a dynamic sequence extended by a single step at each concatenation. It has the recursion
    \[
    h_{n}=\frac{-\mathbf 1^Th^{(n)}+\sqrt{(\mathbf 1^Th^{(n)})^2+8(\mathbf 1^Th^{(n)})+8}}{2},~n\in\N_0,
    \]
    where $\mathbf 1^Th^{(0)}=0$ and $\mathbf 1^Th^{(n)} = \sum_{j=0}^{n-1}h_j$ for $n\geqslant 1$. This
    is exactly the recursive sequence proposed by Teboulle and Vaisbourd in~\cite[Theorem 4]{teboulle2023elementary}, enjoying  the bound $f_{n_k}-f_\star\leqslant \left(\frac14+o(1)\right)\cdot n_k^{-1}\cdot \frac L2 \Vert\mathbf x_0-\mathbf x_\star\Vert^2$.~\cref{eq:ds1} in~\cref{them-dynamic} indicates that the proposed dynamic stepsize improves this result by an arbitrary constant factor by properly selecting $h_b$. For example, taking $h_b=h_\algA^{(1)}=[\sqrt 2]^T\in\mathbb R^1$ yields a dynamic sequence with the bound $f_{n_k}-f_\star\leqslant \left(\frac{1}{2(\sqrt 2+1)}+o(1)\right)\cdot n_k^{-1}\cdot\frac L2\Vert\mathbf x_0-\mathbf x_\star\Vert^2$.~\cref{eq:ds11} shows that if \( h_b \) is generated by~\cref{def-alg1}, then a larger value of \( m \) leads to a better asymptotic bound. We also note that the sequence exhibits a nearly periodic pattern: {starting from iteration \( n_0 + 1 \), every \( m + 1 \) steps of the dynamic sequence consist of a long step  \( \alpha^{(k)} \) followed by the \( m \) steps of \( h_b \)}.
\end{remark}

\section{SSs Minimizing Gradient Norms}\label{sect:H-duality}
{In this section, we introduce a family of SSs tailored to minimizing the gradient norm of the last iterate.} % In this section, we introduce a new SS that improves upon all existing methods in the literature by providing a bound on the gradient norm of the last iterate. %The key component is a novel concatenation technique.
The main workhorse is a new type of concatenation, \texttt{ConGP}, which constructs SSs that adhere to these bounds. Similar to the insights in~\cite{grimmer2024acceleratedobjective}, we highlight a universal ``dual'' relationship between \texttt{ConGP} and the previously discussed \texttt{ConPD}. {A similar phenomenon, referred to as H-duality, was studied earlier in the literature~\cite{kim2021ogm-g, kim2024time}.} To begin, we introduce a new class of SSs.

\begin{definition}[G-bounded SS] \label{def-dual} We say an SS $h\in\mathbb R_+^n$ is \emph{g(radient)-bounded} if it satisfies
\begin{equation}\label{eq:gbd}
    f_0  - f_n\geqslant (\mathbf 1^Th)\cdot \frac{1}{L}\Vert\mathbf g_n\Vert^2,~\forall (f,\mathbf x_0)\in \mathcal C^{1,1}_L\times \mathbb R^d.
\end{equation}
% Define the newcommand \dual in the beginning of document

\end{definition}

\begin{theorem}[Upper bound of \dual SS]
    \label{ineq-h-duality} If an SS $h\in\mathbb R_+^n$ is \dual, then it satisfies
    \[
    \frac{1}{2L}\Vert\mathbf g_n\Vert^2\leqslant \frac{1}{2(\mathbf 1^Th)+1}\cdot (f_0-f_\star).
    \]
    The equality is attained when $\Vert\mathbf x_0 \Vert = 1$ and $f$ is the Huber function defined by
    \begin{equation}\label{eq:huber-grad}
        f(\mathbf x)= \left\{\begin{array}{ll}
        \frac Lw\Vert \mathbf x\Vert-\frac{L}{2w^2}, &  \Vert\mathbf x \Vert\geqslant \frac1w,
        \\ \frac L2\Vert\mathbf x \Vert^2, & \Vert \mathbf x \Vert<\frac1w,\end{array}\right.
    \end{equation}
    where $w =  (\mathbf 1^Th)+1$.\footnote{Note that the selection of $w$ here is slightly different from that of~\cref{them-dominate}.}
\end{theorem}

\begin{proof}
    Recall that $0\geqslant Q_{\star,n} =f_\star -f_n +\frac{1}{2L}\Vert\mathbf g_n\Vert^2$ by~\cref{lemma:Qij}. This, together with~\eqref{eq:gbd}, yields $f_0-f_\star \geqslant \left((\mathbf 1^Th)+\frac{1}{2}\right)\cdot \frac{1}{L}\Vert\mathbf g_n\Vert^2$, leading to the desired bound.

    Next we demonstrate that the equality is attained when $f$ is defined in~\eqref{eq:huber-grad}. It is easy to verify $f\in\mathcal C^{1,1}_L$~\cite[Theorem 6.60]{beck2017first} and has minimizer $\mathbf x_\star = \mathbf 0$. Since $\mathbf x_0 $ is assumed to be a unit vector, a simple induction gives that $\mathbf x_i = \mathbf x_0 -\frac1w \sum_{j=0}^{i-1}h_j \mathbf x_0 $ for $i=1,2,\dots,n$ and $\mathbf x_n = \mathbf x_0 - \frac 1w( \mathbf 1^Th) \mathbf x_0 =\frac{1}{w}\mathbf x_0$. Hence $\Vert\mathbf g_n \Vert= \frac Lw$ and
    \[f_0-f_\star = \frac Lw - \frac{L}{2w^2} -0= \frac{2w-1}{2w^2}L=\left(2(\mathbf 1^Th)+1\right)\cdot \frac{1}{2L}\Vert\mathbf g_n\Vert^2.
    \]
\end{proof}

\subsection{Reverse  Concatenation}

\cref{ineq-h-duality} inspires us to focus on \dual SSs that inherently possess tighter bounds on $\Vert\mathbf g_n\Vert^2$.
The following theorem implies that such SSs can be established recursively through a concatenation process similar to \texttt{ConPD}.

\begin{theorem}[Concatenation of \dual and  primitive SSs (\texttt{ConGP})]\label{them-CDP}
Suppose $h_d\in\mathbb R_+^n$ is a \dual SS and $h_b\in\mathbb R_+^{m-n-1}$ is a
primitive SS, where $m\geqslant n+1$.
Then the new SS $h_e =[h_d^T,\beta,h_b^T]^T\in\mathbb R_+^m$ is also \dual, where $\beta=\psi(\mathbf 1^Th_b,\mathbf 1^Th_d)$ and $\psi(\cdot,\cdot)$ is defined by~\eqref{eq:def-psi}.
%\begin{equation*}%\label{eq:def-psi}
%\psi(x,y) =\frac{3-2y+\sqrt{(2y+1)(2y+8x+9)}}{4}.
%\end{equation*}
\end{theorem}
We denote by $h_e=\Call{ConGP}{h_d,h_b}$ a procedure of \texttt{ConGP} with $h_d$ and $h_b$ as input and $h_e$ as output.

\begin{proof}[Proof of \CDP]
Let $x = \mathbf 1^Th_b$ and $y = \mathbf 1^Th_d$. This gives $\mathbf 1^Th_e = x+\beta+y$.
As $1<\beta<\mathbf 1^Th_b+2$ by~\cref{lemma:CDP} and $h_b$ is primitive,~\cref{lemma:sd} yields
\begin{equation}\label{eq:CDP1}
    f_n - f_m\geqslant \frac{x+3\beta-2\beta^2}{x+2-\beta}\cdot \frac{1}{2L}\Vert\mathbf g_n\Vert^2+\frac{2x^2+3x+\beta}{x+2-\beta}\cdot \frac{1}{2L}\Vert\mathbf g_m\Vert^2.
\end{equation}
%It then remains to verify~\eqref{eq:CDP1} and~\eqref{eq:CDP-target} are equivalent.
Since $\beta = \frac{3-2y+\sqrt{(2y+1)(2y+8x+9)}}{4}$ is  the root of the equation $2\beta^2 -(3-2y)\beta-(2xy+x+4y)=0$, we have
\begin{equation*}%\label{eq:CDP2}
    \frac{x+3\beta - 2\beta^2}{2(x+2-\beta)}+y = \frac{-2\beta^2+(3-2y)\beta +(2xy+x+4y)}{2(x+2-\beta)}=0,~\text{and}
\end{equation*}
\begin{equation*}%\label{eq:CDP3}
    \frac{2x^2+3x+\beta}{2(x+2-\beta)}-(x+\beta+y) = \frac{2\beta^2-(3-2y)\beta -(2xy+x+4y)}{2(x+2-\beta)}=0.
\end{equation*}
This converts~\eqref{eq:CDP1}  to  $f_n -f_m\geqslant -y\cdot \frac 1L \Vert\mathbf g_n\Vert^2+(x+\beta+y) \cdot \frac 1L \Vert\mathbf g_m\Vert^2$, i.e.,  
% \begin{equation*}\label{eq:CDP-target}
% f_n -f_m\geqslant -y\cdot \frac 1L \Vert\mathbf g_n\Vert^2+(x+\beta+y) \cdot \frac 1L \Vert\mathbf g_m\Vert^2.
% \end{equation*}
\begin{equation}\label{eq:CDP0}
f_n -f_m\geqslant -(\mathbf 1^Th_d)\cdot \frac 1L \Vert\mathbf g_n\Vert^2+(\mathbf 1^Th_e)\cdot \frac 1L \Vert\mathbf g_m\Vert^2.
\end{equation}
Since $h_d$ is a  \dual SS, from~\cref{def-dual} we have $f_0-f_n\ge (\mathbf 1^Th_d)\cdot\frac{1}{L}\|\g_n\|^2$.
This adding with~\eqref{eq:CDP0} gives $f_0 -f_m\geqslant (\mathbf 1^Th_e)\cdot \frac 1L \Vert\mathbf g_m\Vert^2$.
So $h_e$ is \dual by~\cref{def-dual}.
%The proof is complete.
\end{proof}

It is noteworthy that the $\beta$ in  \texttt{ConGP} has a symmetric structure to that of $\beta$ in \texttt{ConPD} in~\cref{them-CPD}. Both \texttt{ConPD} and \texttt{ConGP} involve concatenating a primitive SS; however, they differ in the direction of concatenation---the former appends the primitive stepsize at the beginning, while the latter attaches it to the end. Consequently, one might expect a dual  application of  \texttt{ConGP} when utilizing \texttt{ConPD}. This dual nature is further discussed in the next subsection, where we obtain \dual SSs by reversing the dominant SSs.

\subsection{SSs for Small Gradient Norms}
%Given the analogous structure between \CDP\ and \CPD, we can readily adopt the idea of recursive concatenation in Definition \ref{def-alg2} to construct a new family of \dual SSs with favorable  performance in terms of gradient norm estimation $\frac{1}{2L}\Vert\mathbf g_n\Vert^2\leqslant C\cdot (f_0-f_\star)$.

We next define a new family of SSs by simply reversing the entries of $h_\algB^{(n)}$.
{Given an SS $h=[h_0,\dotsc,h_{n-1}]^T$, let $\widetilde h$ denote the reverse of $h$, i.e., $\widetilde h = [h_{n-1},\dotsc,h_0]^T$.}
\begin{definition}\label{def-alg3}
Define a family of SSs $H_\algC=\{h_\algC^{(i)}:\ i\in\N_0\}$ as follows,
\begin{equation*}%\label{eq:def-alg3-reverse}
h_\algC^{(n)}={\tilde h_\algB^{(n)}}\in\mathbb R^n,
\end{equation*}
where $h_\algB^{(n)}$ is the SS defined in Definition \ref{def-alg2}.
\end{definition}

The next proposition reveals that the SS defined in~\cref{def-alg3} is \dual.
\begin{proposition}\label{prop:def-alg3-recur} Any SS $h_\algC^{(n)}$ defined in Definition \ref{def-alg3} is \dual. Also, for $n\geqslant 1$,  $h_\algC^{(n)}$ satisfies the recursion:
\begin{equation}\label{eq:def-alg3}
\begin{aligned}
\mathcal H_\algC^{(n)}&=\left\{ \Call{ConGP}{h_\algC^{(k)}, {\tilde h_\algA^{(n-k-1)}}}: \ k=0,1,\dotsc,n-1 \right\}\subset \mathbb R^n,
   \\ h_\algC^{(n)}&\in\argmax_{h}\left\{\mathbf 1^Th:\ h\in\mathcal H_\algC^{(n)}\right\}.%\in \mathbb R^n,
\end{aligned}
\end{equation}
where $h_\algA^{(n)}$ is the primitive SS defined in Definition \ref{def-alg1}.% and ${\tilde h_\algA^{(n)}}$ is its reversed vector.

\end{proposition}

\begin{proof} To ensure that~\eqref{eq:def-alg3} is well-defined, we  first show that ${\tilde h_\algA^{(n)}}$, $n\in\N_0$, are also primitive. For $n=0$,  $[~]^T\in\mathbb R^0$ is the reverse of itself, which is primitive. We then proceed by induction. Assume ${\tilde h_\algA^{(k)}}$ is primitive when $k<n$.
Recall that  there exists $k<n$ such that $h_\algA^{(n)}= \Call{ConPP}{h_\algA^{(k)},h_\algA^{(n-k-1)}}$. By induction $\Call{ConPP}{{\tilde h_\algA^{(n-k-1)}},{\tilde h_\algA^{(k)}}}$ should also be  primitive as an output of \texttt{ConPP}. But it is exactly ${\tilde h_\algA^{(n)}}$ by checking the definition of \texttt{ConPP}, which completes the induction.
% Note that the stepsize $\alpha$ defined in~\cref{them-CPP} is equal for $\Call{ConPP}{h_\algA^{(k)},h_\algA^{(n-k-1)}}$ and $\Call{ConPP}{{\tilde h_\algA^{(n-k-1)}},{\tilde h_\algA^{(k)}}}$, so we have $\Call{ConPP}{{\tilde h_\algA^{(n-k-1)}},{\tilde h_\algA^{(k)}}}={\tilde h_\algA^{(n)}}$. Thus, ${\tilde h_\algA^{(n)}}$ is primitive, completing the induction.

Given that every  ${\tilde h_\algA^{(n)}}$ is primitive, we now prove our proposition. When $n=0$, $h_\algC^{(0)} = [~]^T\in\mathbb R^0$ is trivially \dual because $f_0 - f_0  \geqslant 0\cdot \Vert\mathbf g_0\Vert^2$. Fix $n\geqslant 1$ and we assume ${h_\algC^{(k)}}$ is \dual for $k<n$. By the definitions of \texttt{ConPD} and \texttt{ConGP}, the reversed vector of $\Call{ConPD}{h_\algA^{(k)}, h_\algB^{(n-k-1)}}$ coincides with $\Call{ConGP}{{\tilde h_\algB^{(n-k-1)}}, {\tilde h_\algA^{(k)}}}$, which is just $\Call{ConGP}{{h_\algC^{(n-k-1)}}, {\tilde h_\algA^{(k)}}}$. Therefore, the collection $\mathcal H_\algC^{(n)}$ contains all reversed vectors of $\mathcal H_\algB^{(n)}$. As $h_\algB^{(n)}$ is the SS with maximum $\mathbf 1^Th$ in $\mathcal H_\algB^{(n)}$, its reverse, $h_\algC^{(n)}$, has the maximum value $\mathbf 1^Th$ among $\mathcal H_\algC^{(n)}$. So~\eqref{eq:def-alg3} holds by induction.
\end{proof}

As~\cref{prop:def-alg3-recur} asserts that every $h_\algC^{(n)}$ is \dual, the following worst-case bound on  gradient norm directly  follows from~\cref{ineq-h-duality}.

\begin{proposition}\label{them-alg3}Gradient descent~\eqref{eq:gd}
 with SS $h_\algC^{(n)}$ has
\[\frac{1}{2L}  \Vert\mathbf g_n\Vert^2\leq \frac{1}{2(\mathbf 1^Th_\algC^{(n)})+1}\cdot (f_0-f_\star),\]
where equality is attained by the Huber function~\eqref{eq:huber-grad}.
\end{proposition}

Comparing Propositions \ref{them-alg2} and \ref{them-alg3}, there is a symmetry between the families of SSs  $H_\algB$ and $H_\algC$, which are effective in reducing the objective gap and the gradient norm, respectively.
For example, the schedule $h_\algB^{(5)}=[\sqrt 2,2,\sqrt 2,\sqrt{7+4\sqrt 2},3/2]^T$ corresponds to the objective bound $f_5-f_\star\leqslant 0.04814\cdot \frac L2\Vert\mathbf x_0-\mathbf x_\star\Vert^2$, while $h_\algC^{(5)}=[3/2,\sqrt{7+4\sqrt 2},\sqrt 2,2,\sqrt 2]^T$ corresponds to the gradient bound  $  \frac{1}{2L}\Vert \mathbf g_5\Vert^2\leqslant 0.04814 (f_0-f_\star)$.
As a consequence, the asymptotic behavior of $\mathbf 1^Th_\algB^{(n)}$ studied in~\cref{sect:asymptotic} also applies to $\mathbf 1^Th_\algC^{(n)}$,
% So there is also an $O(n^{-\varrho})$ convergence rate for bounds of $h_\algC^{(n)}$, 
as presented in the next theorem.

\begin{theorem} \label{them-convergence3.29}
Let $\bar\nu_\algC:=\liminf_{n\to\infty}\mathbf 1^Th_\algC^{(n)}/n^\varrho$ and $\bar\nu_\algB:=\liminf_{n\to\infty}\mathbf 1^Th_\algB^{(n)}/n^\varrho $.  Then $\bar\nu_\algC=\bar\nu_\algB$. 
Gradient descent~\eqref{eq:gd} with $h_\algC^{(n)}$ enjoys an asymptotic bound of
    \[
    \frac{1}{2L}\Vert\mathbf g_n\Vert^2\leqslant ((2\bar\nu_\algC)^{-1}+o(1))\cdot n^{-\varrho}\cdot  (f_0-f_\star),
    \]
    where $\varrho =\log_2(\sqrt 2+1)$.
    % and $(2\bar\nu_\algC)^{-1}<0.4239$.
\end{theorem}

\begin{proof} 
Since $\mathbf 1^Th_\algC^{(n)}=\mathbf 1^Th_\algB^{(n)}$ by the definition of $h_\algC^{(n)}$, it follows that $\bar\nu_\algB=\bar\nu_\algC$. The asymptotic bound is a direct corollary of~\cref{them-alg3}.
\end{proof} 

\begin{remark}\label{them-convergence3.3}
{Combining~\cref{them-convergence3.29} with the numerical evidence from~\cref{them-convergence2.3} suggests that $(2\bar\nu_\algC)^{-1}<0.4239$.} 
This improves~\cite{grimmer2024acceleratedobjective} asymptotically by a constant factor $\frac{0.4302}{0.4239}\approx 1.015$. Moreover, our SS is available for any iteration number $n$, while the one in~\cite{grimmer2024acceleratedobjective} is only available for $n=2^l-1,~l\in\N$.
\end{remark}

% From~\cref{them-convergence2.3}, we have $\liminf_{n\to\infty}\mathbf 1^Th_\algB^{(n)}/n^\varrho\geqslant 1.1796$. This  implies that lower bound $\liminf_{n\to\infty}\mathbf 1^Th_\algC^{(n)}/n^\varrho\geqslant 1.1796$ since $\mathbf 1^Th_\algC^{(n)}=\mathbf 1^Th_\algB^{(n)}$ by the definition of $h_\algC^{(n)}$. This, together with~\cref{them-alg3}, yields the desired rate.

\begin{remark}
We note that a concurrent work by Grimmer, Shu, and Wang~\cite{grimmercomposing,grimmer2025composing-MOR}---first released as an \textit{arXiv} preprint on October~21,~2024, and later published in \textit{Mathematics of Operations Research}---shares conceptual similarities with our approach, which appeared on \textit{arXiv} five days earlier, on October~16,~2024. 
In their work, Grimmer et al.~\cite{grimmercomposing,grimmer2025composing-MOR} proposed a composition framework that is analogous to our method. Specifically, their notions of s-composable, f-composable, and g-composable SSs closely parallel our definitions of primitive, dominant, and g-bounded  SSs. 
{
A notable methodological difference is that~\cite{grimmercomposing,grimmer2025composing-MOR} requires tightness on both quadratic and Huber functions as a design constraint, whereas our analysis focuses solely on tightness with respect to the Huber function. Despite this difference in formulation,  our constructed $h_\algA$, $h_\algB$ and $h_\algC$ are identical to their optimized basic stepsize schedules $h^{\texttt{OBS-S}}$, $h^{\texttt{OBS-F}}$ and $h^{\texttt{OBS-G}}$, respectively.}
% While their f- and g-composable SSs are motivated by balancing worst-case performance on Huber and quadratic functions—leading to the constraint  
% $\frac{1}{1+2(\mathbf{1}^T h)} = \prod_{i=0}^{n-1}(h_i - 1)^2$,  
% we show that the consideration of quadratic functions is unnecessary. Specifically, our dominant and g-bounded schedules achieve worst-case performance on Huber functions but not necessarily on quadratic functions.
\end{remark}

\subsection{Dynamic SS for Reducing Gradient Norm}%\label{sect:dynamic2}

Similar to~\cref{them-dynamic} in~\cref{sect:dynamic}, a dynamic stepsize sequence
 reducing the gradient norm is available by applying \texttt{ConGP} repetitively, as illustrated below.

\begin{proposition}[Dynamic Stepsize 2]\label{them-dynamic2} Suppose $h^{(0)}\in\mathbb R_+^{n_0}$ is a \dual SS  and $h_b\in\mathbb R_+^{m}$ is a primitive SS. Define the recursion $h^{(k)} = \Call{ConGP}{h^{(k-1)}, h_b}\in \mathbb R_+^{n_k}$ for $k\in \N$. Then $n_k= n_0+k(m+1)$ and $\lim_{k\to\infty}\mathbf 1^Th^{(k)} /n_k =2(\mathbf 1^Th_b+1)/(m+1)$. Moreover, gradient descent~\eqref{eq:gd} with stepsize sequence $h^{(k)}$ enjoys the bound
\begin{equation}\label{eq:ds2}
\frac{1}{2L}\Vert \mathbf g_{n_k}\Vert^2\leqslant \frac{1}{2(\mathbf 1^Th^{(k)})+1}\cdot (f_0-f_\star)\simeq \frac{m+1}{4n_k(\mathbf 1^Th_b+1)}\cdot (f_0-f_\star).
\end{equation}
In particular, if $h_b=h_\algA^{(m)}$, we have
\begin{equation}\label{eq:ds21}
\frac{1}{2L}\Vert \mathbf g_{n_k}\Vert^2\leqslant
\left(\frac{m+1}{4(\sqrt 2-1)(m+2)^{\varrho} }+o(1)\right)\cdot \frac{1}{n_k}\cdot (f_0-f_\star).
\end{equation}
\end{proposition}
\begin{proof} The proof mimics that of~\cref{them-dynamic}. For the sake of completeness, we present it here. Recall the recursion  $h^{(k+1)}=[(h^{(k)})^T,\beta^{(k)},h_b^T]^T$ by \CDP, where $\beta^{(k)}=\psi(\mathbf 1^Th_b,\mathbf 1^Th^{(k)})$ and~\eqref{eq:def-psi}. This gives $h^{(k)}\in\mathbb R^{n_0+k(m+1)}$ and the recursion 
% \begin{equation*}\small
% % \label{eq:dynamic2}
% \begin{aligned}
%     \mathbf 1^Th^{(k+1)} 
%     % &= \mathbf 1^Th^{(k)} + \beta^{(k)} + \mathbf 1^Th_b\\ 
% & = \mathbf 1^Th^{(k)}+\frac{3-2(\mathbf 1^Th^{(k)})+\sqrt{(2(\mathbf 1^Th^{(k)})+1)(2(\mathbf 1^Th^{(k)})+8(\mathbf 1^Th_b)+9)}}{4}+\mathbf 1^Th_b.
% \end{aligned}
% \end{equation*}
$ \mathbf 1^Th^{(k+1)} 
    % &= \mathbf 1^Th^{(k)} + \beta^{(k)} + \mathbf 1^Th_b\\ 
  = \mathbf 1^Th^{(k)}+\frac{3-2(\mathbf 1^Th^{(k)})+\sqrt{(2(\mathbf 1^Th^{(k)})+1)(2(\mathbf 1^Th^{(k)})+8(\mathbf 1^Th_b)+9)}}{4}+\mathbf 1^Th_b$, implying 
%~\cref{eq:dynamic2} ensures that $\mathbf 1^Th^{(k+1)}\geqslant \mathbf 1^Th^{(k)}+\mathbf 1^Th_b$, so
% Clearly, $\{\mathbf 1^Th^{(k)}\}_{k=0}^\infty$ is monotonically increasing and unbounded. Applying the Stolz-Ces\`{a}ro Theorem, we obtain
% \begin{equation}\label{eq:dynamic22}
% \begin{aligned}  &\quad\ \lim_{k\to\infty}\frac{\mathbf 1^Th^{(k)} }{k}
% =
%     \lim_{k\to\infty}\frac{\mathbf 1^Th^{(k+1)}-\mathbf 1^Th^{(k)} }{(k+1)-k}\\
% &= \lim_{\mathbf 1^Th^{(k)} \to+\infty} \frac{3-2(\mathbf 1^Th^{(k)})+4(\mathbf 1^Th_b)+\sqrt{(2(\mathbf 1^Th^{(k)})+1)(2(\mathbf 1^Th^{(k)})+8(\mathbf 1^Th_b)+9)}}{4}
% \\ &= \lim_{\mathbf 1^Th^{(k)} \to+\infty} \frac{3-2(\mathbf 1^Th^{(k)})+4(\mathbf 1^Th_b)+(2(\mathbf 1^Th^{(k)})+4(\mathbf 1^Th_b)+5+o(1))}{4}
% \\ &=2(\mathbf 1^Th_b)+2.
%     \end{aligned}
% \end{equation}
 $\mathbf 1^Th^{(k+1)}=\mathbf 1^Th^{(k)}+2(\mathbf 1^Th_b+1)+o(1)$ when $k\to\infty$. By the Stolz-Ces\`{a}ro theorem,  $\lim_{k\to\infty}\mathbf 1^Th^{(k)}/k = 2(\mathbf 1^Th_b+1)$. 
Combining it with $n_k = n_0+k(m+1)$ yields $\lim_{k\to\infty} \mathbf 1^Th^{(k)}/n_k = 2(\mathbf 1^Th_b+1)/(m+1)$ as desired. Then~\eqref{eq:ds2} follows directly from~\cref{ineq-h-duality}.
Equation~\eqref{eq:ds21} follows from~\eqref{eq:ds2} and~\cref{them-convergence1.2}.
\end{proof}

\begin{remark}
    When taking $h^{(0)}= h_b = [~]^T\in\mathbb R^0$ to be the zero-iteration SS, the dynamic stepsize sequence in~\cref{them-dynamic2} extends recursively as follows:
    \begin{equation}\label{eq:remark-rotaru1}
    h_n = \frac{3-2(\mathbf 1^Th^{(n)}) +\sqrt{(2(\mathbf 1^Th^{(n)})+1)(2(\mathbf 1^Th^{(n)})+9)}}{4},~n\in\N_0,
    \end{equation}
    % \todo{define $h_n$}
    where $\mathbf 1^Th^{(0)}=0$ and $\mathbf 1^Th^{(n)} = \sum_{j=0}^{n-1}h_j$ for $n\geqslant 1$.  We can rewrite~\eqref{eq:remark-rotaru1} to express $\mathbf 1^Th^{(n)}$ in terms of $h_n$, yielding the identity $ \mathbf 1^Th^{(n)}=\frac{2h_n^2-3h_n}{4-2h_n}$, $n\in\N_0$. It then follows that %\todo{?}
    % \begin{equation}\label{eq:remark-rotaru2}
    %    \mathbf 1^Th^{(n)}=\mathbf 1^Th^{(n-1)}+h_{n-1}
    %    =\frac{2h_{n-1}^2-3h_{n-1}}{4-2h_{n-1}}+h_{n-1}
    %    =\frac{h_{n-1}}{4-2h_{n-1}},~n\in\N.
    % \end{equation}
    $   \mathbf 1^Th^{(n)}=\mathbf 1^Th^{(n-1)}+h_{n-1}
       =\frac{2h_{n-1}^2-3h_{n-1}}{4-2h_{n-1}}+h_{n-1}
       =\frac{h_{n-1}}{4-2h_{n-1}}$ for $n\in\N$.
    Substituting it into~\eqref{eq:remark-rotaru1} yields
$
        h_n = \frac{3-2h_{n-1}+\sqrt{9-4h_{n-1}}}{2(2-h_{n-1})},~n\in\N.
$
    The resulting recursion is identical to the sequence in~\cite[Corollary 2.19]{rotaru2024exact}. Applying~\cref{them-dynamic2}, this stepsize sequence has bound $\frac{1}{2L}\Vert\mathbf g_{n_k}\Vert^2\leqslant \left(\frac14+o(1)\right)\cdot {n_k}^{-1}\cdot (f_0-f_\star)$.  Similar to~\cref{remark:teboulle}, the constant $1/4$ can be improved by an arbitrary constant factor with proper $h_b$, and  the sequence exhibits a nearly periodic pattern.
\end{remark}

\section{Comparisons of Newly Computed Guarantees with the Literature}\label{sec:experiment}

% In this section, we show that both the objective value and gradient norm bounds of our proposed SSs $h_\algB^{(n)}$ and $h_\algC^{(n)}$ outperform existing {fixed-step first-order} methods in a non-asymptotic sense.
{In this section, we present comparisons of the worst-case bounds achieved by our proposed SSs $h_\algB^{(n)}$ and $h_\algC^{(n)}$ against other memoryless fixed-step SSs reported in the literature.}

\subsection{Comparison for Objective Values}

In this subsection, we first present the approximate numerical values of $h_\algB^{(n)}$ for $1\leqslant n\leqslant 7$ in~\cref{tab:Hb-numer} and supplement the exact values {\cbll in~\cref{tab:Hb-exact}}.
% Table B1 in our \textit{arXiv} version~\cite{zhang2024accelerated} for reference. 
We remark that our SSs are not unique as discussed in~\cref{sect:nonunique}.

\begin{table}[h]\small
    \caption{Numerical SSs of Definition~\ref{def-alg2} for $1\leqslant n\leqslant 7$ with the value $\beta$ of \CPD\  highlighted.}
    \label{tab:Hb-numer}
     \centering
    \begin{tabular}{c|l}%{\textwidth} {@{}c@{\hspace{0.5cm}}l}
    \toprule
        $n$ & $(h_\algB^{(n)})^T$ (Numerical) \\ \midrule
        1 &  [\textbf{1.500000}]\\
        2 &  [1.414214, \textbf{1.876768}]\\
        3 &   [1.414214, \textbf{2.414214}, 1.500000]\\
        4 &   [1.414214, 1.601232, \textbf{3.005144}, 1.500000]\\
        5 &  [1.414214, 2.000000, 1.414214, \textbf{3.557647}, 1.500000]\\
        6 &  [1.414214, 2.000000, 1.414214, \textbf{4.172876}, 1.414214, 1.876768]\\
        7 &   [1.414214, 1.601232, 2.260578, 1.414214, \textbf{4.826959}, 1.414214, 1.876768]\\
    \bottomrule
    \end{tabular}
\end{table}

{Next we compare the objective value bounds of our proposed SSs family $H_\algB$ from~\cref{def-alg2} with existing {memoryless fixed-step SSs} in~\cref{tab:compare-combined} (left five columns) of Teboulle and Vaisbourd~\cite[Theorem 4]{teboulle2023elementary}, Das Gupta et al.~\cite{gupta2024branch}, Kamri et al.~\cite{KHG25} and Grimmer et al.~\cite{grimmer2024acceleratedobjective}. Specifically, we analyze the constant $C$ in the convergence bound $f_n-f_\star\leqslant C\cdot  \frac L2\Vert\mathbf x_0-\mathbf x_\star\Vert^2$, where~\cite{teboulle2023elementary} and~\cite{grimmer2024acceleratedobjective} provide analytic solutions, while the bounds of~\cite{gupta2024branch} and~\cite{KHG25} are numerical. Particularly, we numerically computed the bounds of~\cite{gupta2024branch} by PEPit~\cite{pepit2022} with SSs provided in the supplementary material of~\cite{gupta2024branch}. For~\cite{KHG25}, we obtain the bounds from Table~1 in~\cite{KHG25}. Since~\cite{KHG25} only reported on iterations where their method outperforms~\cite{gupta2024branch}, the unmentioned iterations are omitted. We also note that the metric in~\cite{KHG25} should be scaled by a factor of $2$ to align with our metric.}

% Next we compare the objective value bounds of our proposed SSs family $H_\algB$ from~\cref{def-alg2} and the objective bound in~\cref{them-alg2} with those from various sources in~\cref{tab:compare-combined}: the dynamic SS in Teboulle and Vaisbourd~\cite[Theorem 4]{teboulle2023elementary}, the SSs in Grimmer et al.~\cite{grimmer2024acceleratedobjective}, and the numerical SSs reported by Das Gupta et al.~\cite{gupta2024branch} for different iteration numbers.  It is worth noting that the bounds in~\cite{teboulle2023elementary} and~\cite{grimmer2024acceleratedobjective}, and ours are analytic, while the ones in~\cite{gupta2024branch} are obtained numerically by PEPit~\cite{pepit2022} with stepsizes provided in the supplementary material of~\cite{gupta2024branch}.
% Specifically, we compare the constant $C$ in the upper bound estimation given by  $f_n-f_\star\leqslant C\cdot  \frac L2\Vert\mathbf x_0-\mathbf x_\star\Vert^2$ across different methods.

\begin{table}[t]
    \centering\small
\caption{Convergence bounds comparison for SSs, evaluating both objective value  and gradient norm. 
Shared columns (center) present identical bounds for both metrics but with different SSs. Best results per metric are bolded.}\label{tab:compare-combined}
\begin{tabular*}{\textwidth}{@{\extracolsep\fill}ccccccc}\toprule
\multirow{2}{*}{Iterations} & \multicolumn{3}{c}{Objective Value} & \multicolumn{2}{c}{Shared ($f$/$g$)} & Gradient Norm    \\ \cmidrule(lr){2-4} \cmidrule(lr){5-6} \cmidrule(lr){7-7}
& TV~\cite{teboulle2023elementary}    & DVR~\cite{gupta2024branch}  & KHG~\cite{KHG25} & GSW~\cite{grimmer2024acceleratedobjective}      & Ours      & RGP~\cite{rotaru2024exact}  \\ \midrule
1 & 0.261204 & \textbf{0.250000} & \slash & \textbf{0.250000} & \textbf{0.250000} & \textbf{0.250000} \\
2 & 0.142229 & \textbf{0.131892} & \slash & \slash & \textbf{0.131892} & 0.133975 \\
3 & 0.095827 & \textbf{0.085786} & \slash & \textbf{0.085786} & \textbf{0.085786} & 0.090059 \\
4 & 0.071613 & \textbf{0.062340} & \slash & \slash & \textbf{0.062340} & 0.067412 \\
5 & 0.056899 & \textbf{0.048141} & \slash & \slash & \textbf{0.048141} & 0.053707 \\
6 & 0.047070 & 0.040197 & 0.039790 & \slash & \textbf{0.039086} & 0.044561 \\
7 & 0.040066 & \textbf{0.032662} & \slash & 0.032768 & \textbf{0.032662} & 0.038039 \\
8 & 0.034835 & 0.028109 & 0.027924 & \slash & \textbf{0.027869} & 0.033161 \\
9 & 0.030787 & 0.024565 & 0.024368 & \slash & \textbf{0.024182} & 0.029378 \\
10 & 0.027565 & \textbf{0.021245} & \slash & \slash & \textbf{0.021245} & 0.026362 \\
11 & 0.024943 & 0.019184 & \slash & \slash & \textbf{0.018869} & 0.023902 \\
12 & 0.022768 & 0.017282 & \slash & \slash & \textbf{0.016986} & 0.021858 \\
13 & 0.020936 & 0.015969 & 0.015892 & \slash & \textbf{0.015422} & 0.020133 \\
14 & 0.019373 & 0.014752 & 0.014408 & \slash & \textbf{0.014098} & 0.018658 \\
15 & 0.018024 & 0.013184 & \slash & 0.013082 & \textbf{0.012959} & 0.017384 \\
25 & 0.010587 & 0.006952 & \slash & \slash & \textbf{0.006872} & 0.010308 \\
31 & 0.008473 & 0.005443 & 0.005364 & 0.005327 & \textbf{0.005264} & 0.008279 \\
63 & 0.004088 & \slash & \slash & 0.002189 & \textbf{0.002159} & 0.004031 \\
127 & 0.002003 & \slash & \slash & 0.000903 & \textbf{0.000890} & 0.001987 \\
255 & 0.000990 & \slash & \slash & 0.000373 & \textbf{0.000368} & 0.000986 \\
511 & 0.000492 & \slash & \slash & 0.000155 & \textbf{0.000152} & 0.000491 \\\bottomrule
\end{tabular*}
\end{table} 
{\cref{tab:compare-combined} shows that our analytically derived stepsize schedules are the best-known memoryless fixed-step SSs for the compared values of $n$.}
{Notably}, our SSs align precisely with the optimal numerical SS in~\cite{gupta2024branch}  for total iteration numbers $n=1,2,3,4,5,7$, $10$ and surpass~\cite{gupta2024branch}  for other total iteration numbers, exhibiting an absolute discrepancy no larger than $10^{-6}$.
We thus \emph{conjecture} that the SSs generated by our algorithm defined in Definition \ref{def-alg2} are indeed globally optimal.
Additional support for this conjecture comes from the case where $n=2$, where our derived $h_\algB^{(2)} = [\sqrt{2}, (3 + \sqrt{9 + 8\sqrt{2}})/4]^T$ matches the one derived in~\cite[Section 4.2.1]{daccache2019performance}.
From the table, we should also note that  the SSs reported in~\cite{gupta2024branch} may not be actually optimal as the branch and bound method may stop early and there exists numerical error in practical implementation, which is evident from the unhighlighted columns for methods in~\cite{gupta2024branch}.

%Furthermore, for other numbers of iterations, our results consistently surpass~\cite{gupta2024branch}, which we cautiously attribute to the numerical instability of their nonconvex optimization methodology.

For the SSs in~\cite{teboulle2023elementary, silver2, grimmer2024acceleratedobjective}, they are encompassed as special cases of our concatenation techniques as demonstrated in Remarks \ref{remark:silver} and \ref{remark:teboulle}.
As a consequence, our constructed sequences $h_\algB^{(n)}$ exhibit a consistent advantage.
As~\cite{grimmer2024acceleratedobjective} can only provide SS for $n=2^l-1$ {\cbl($l\in\N$)}, we present only the associated bounds in~\cref{tab:compare-combined}.
Although the SSs in~\cite{grimmer2024acceleratedobjective} coincide with $h_\algB^{(n)}$ for $n=1$ and $n=3$, theirs are surpassed by ours afterwards.

Finally, we remark that the branch and bound method in~\cite{gupta2024branch} is very time consuming, and impractical for large $n$. As reported in their paper, it costs 1 day and 18 hours for $n=10$ on a standard laptop, and 2 days and 20 hours for $n=25$ on \texttt{MIT Supercloud}. In contrast, {our method is practically efficient for moderate values of $n$}, requiring less than 1 second for $n = 10^4$ and 4 minutes for $n = 10^5$ to compute the SS on a standard laptop.

\subsection{Comparison for Gradient Norms}
% We next compare the gradient norm bounds for the family of SSs $H_\algC$ from~\cref{def-alg3} with other methods.
% Specifically, we examine the constant  $C$ in the upper bound estimation given by
% $\frac{1}{2L} \Vert\mathbf g_n\Vert^2\leqslant C\cdot  (f_0-f_\star)$ across different methods.
% In~\cref{tab:compare-combined}, we present the comparison of the gradient bounds from~\cref{them-alg3} with the dynamic stepsize sequence by Rotaru et al.~\cite[Corollary 2.19]{rotaru2024exact} and the SSs in Grimmer et al.~\cite{grimmer2024acceleratedobjective}. All associated bounds are derived analytically. We remark that~\cite{grimmer2024acceleratedobjective} only present SSs for $n = 2^l-1$.

For gradient norm bounds in~\cref{tab:compare-combined} (right three columns), we examine the constant $C$ in $\frac{1}{2L} \Vert\mathbf g_n\Vert^2\leqslant C\cdot  (f_0-f_\star)$ of SSs from Rotaru et al.~\cite[Corollary 2.19]{rotaru2024exact}, Grimmer et al.~\cite{grimmer2024acceleratedobjective}, and our $H_\algC$ in~\cref{def-alg3}. Notably, although the SS constructions differ from those in the objective value comparison, the displayed constants $C$ for~\cite{grimmer2024acceleratedobjective} and our method are identical to their respective objective-bound counterparts in~\cref{tab:compare-combined}. This duality phenomenon, discussed in~\cref{sect:H-duality}, causes the shared columns between the two metric groups in the table. As shown in the table, {our method consistently {\cbll achieves better performance than} all other compared methods for all compared iteration numbers in this metric}. 

\section{Conclusion}\label{sec:con}
We introduce three special classes of SSs for smooth convex gradient descent: primitive, dominant, and \dual. We show that new primitive and dominant SSs can be constructed via simple concatenation of {two SSs with fewer steps}. Using these techniques, we develop two types of SSs, $h_\algB$ and $h_\algC$, with state-of-the-art worst-case complexity bounds, which consistently outperform {existing memoryless fixed-step gradient descent methods} in terms of objective values and gradient norms.
Notably, our analytically derived SSs generally surpass the previously numerically computed ``optimal'' SSs and, in other cases, they match exactly within a certain level of numerical precision. Whether these SSs are truly optimal {for memoryless fixed-step gradient descent} remains an open question.

% \backmatter

\appendix

\section{Technical Lemmas}
In this appendix, we present lemmas supporting the proofs of the main results.
\begin{lemma}\label{lemma:derivative}
    For $x,y\geqslant 0$, $\left(1+x/(x^{\frac 1\varrho}+y^{\frac 1\varrho})^\varrho \right)\left(1+y/(x^{\frac 1\varrho}+y^{\frac 1\varrho})^\varrho\right)\leqslant 2$  where $\varrho=\log_2(\sqrt 2+1)$.
\end{lemma}
\begin{proof}[Proof of Lemma \ref{lemma:derivative}] Let $a = x^{\frac 1\varrho}/(x^{\frac 1\varrho}+y^{\frac 1\varrho})\in [0,1] $. Then it is equivalent to showing $(1+a^\varrho)(1+(1-a)^\varrho)\leqslant 2$. Without loss of generality, we assume $a\in [0,1/2]$. Define $\zeta (a) = \log (1+a^\varrho)+\log(1+(1-a)^\varrho)-\log 2$, then $\zeta'(a) = \frac{\varrho a^{\varrho-1}}{1+a^\varrho}- \frac{\varrho (1-a)^{\varrho-1}}{1+(1-a)^\varrho}$. We aim to study the sign of $\zeta'(a)$ to prove $\zeta(a)\leqslant 0$ over $a\in [0,1/2]$ .

Define $\xi(a) = \frac{1+a^\varrho}{a^{\varrho-1}} - \frac{1+(1-a)^\varrho}{(1-a)^{\varrho-1}} = a^{1-\varrho}-(1-a)^{1-\varrho} +2a-1$. Taking the derivative, we get $\xi'(a) = (1-\varrho)(a^{-\varrho}+(1-a)^{-\varrho})+2$ and $\xi''(a) = \varrho(\varrho-1)(a^{-\varrho-1} - (1-a)^{-\varrho-1})$. It is then clear that $\xi''(a)\geqslant 0$ when $a\in (0,\frac12)$, so $\xi'(a)$ increases over the interval. As $\lim_{a\to 0_+}\xi'(a)=-\infty$ and $\xi'(\frac12)>0$, we learn that $\xi(a)$ initially decreases and then increases on $(0,\frac12)$. While $\lim_{a\to 0_+}\xi(a)=+\infty$, $\xi(\frac14)<0$ and $\xi(\frac12)=0$, there exists $a_0\in (0,\frac14)$ such that $\xi(a)\geqslant 0$ when $a\in (0,a_0]$, and $\xi(a)\leqslant 0$ when $a\in [a_0,\frac12]$.

As a consequence, $\zeta'(a)\leqslant 0$ when $a\in (0,a_0]$ and $\zeta'(a)\geqslant 0$ when $a\in [a_0,\frac12]$. This concludes that $\zeta(a)\leqslant \max\{\zeta(0),\zeta(\frac12)\}=0$, thereby completing our proof.
\end{proof}

\begin{lemma}\label{lemma:2.376373}
 Let $\mu\in (0,1)$, $\varrho>0$ and $w\geqslant 0$. Then the inequality  $2(1-\mu)^\varrho + \frac12w\mu^\varrho + \sqrt{2w\mu^\varrho(1-\mu)^\varrho + \frac14w^2\mu^{2\varrho}} \leqslant w$ holds if and only if $w\geqslant \frac{2(1-\mu)^\varrho}{1-\mu^{\varrho/2}}$ holds.
\end{lemma}
\begin{proof}[Proof of Lemma \ref{lemma:2.376373}]
Let $x=2(1-\mu)^\varrho\in(0,2)$ and $y=\frac12\mu^\varrho\in (0,\frac12)$. Then it suffices  to show that $x+wy+\sqrt{2wxy + w^2y^2}\leqslant w$ holds if and only if $w\geqslant x/(1-\sqrt{2y})$.

First assume $x+wy+\sqrt{2wxy + w^2y^2}\leqslant w$. Then we have the quadratic inequality $(w-wy-x)^2\geqslant 2wxy+w^2y^2$, i.e.,  $(1-2y)w^2-2xw+x^2\geqslant 0$, implying $w\geqslant \frac{x+\sqrt{2x^2y}}{1-2y}$ or $w\leqslant \frac{x-\sqrt{2x^2y}}{1-2y}$. Since
$w\geqslant x+wy+\sqrt{2wxy + w^2y^2}\geqslant x+wy+wy=x+2wy$, it requires $w\geqslant \frac{x}{1-2y}$. This implies  $w\geqslant \frac{x+\sqrt{2x^2y}}{1-2y} = x\cdot \frac{1+\sqrt{2y}}{1-2y} = \frac{x}{1-\sqrt{2y}}$.

Now suppose $w\geqslant \frac{x}{1-\sqrt{2y}}$. Then it also leads to the quadratic inequality  $(w-wy-x)^2\geqslant 2wxy+w^2y^2$. Moreover, $w - wy-x \geqslant x\cdot \frac{\sqrt{2y}-y}{1-\sqrt{2y}}\geqslant 0$ because $w\geqslant \frac{x}{1-\sqrt{2y}}$ and $y\in (0,\frac12)$. Hence we conclude $w-wy-x\geqslant\sqrt{2wxy+w^2y^2}$.
\end{proof}

\begin{lemma}\label{lemma:CTP} For $x,y\geq 0$ and  $\varphi(\cdot,\cdot)$   in~\eqref{eq:def-varphi}, it holds that $1< \varphi(x,y)  <y+2$.
\end{lemma}
\begin{proof}
% On the one hand, we have $\varphi(x,y)> \frac{-x-y+(x+y+2)}{2}=1$. On the other hand, the identity $
%     (x+y+2)^2+4(x+1)(y+1)=(x+3y+4)^2-8(y+1)^2$
% yields $\varphi(x,y)<\frac{-x-y+(x+3y+4)}{2}=y+2$.
Because $   (x+y+2)^2+4(x+1)(y+1)=(x+3y+4)^2-8(y+1)^2$, we have that $1=\frac{-(x+y)+(x+y+2)}{2}<\varphi(x,y)<\frac{-(x+y)+(x+3y+4)}{2}=y+2$.
\end{proof}

 \begin{lemma}\label{lemma:CDP}
 For $x,y\geq 0$ and $\psi(\cdot,\cdot)$   in~\eqref{eq:def-psi}, it holds that $1< \psi(x,y)  <x+2$.
\end{lemma}
\begin{proof}
% On the one hand, $\frac{3-2y+\sqrt{(2y+1)(2y+8x+9)}}{4}> \frac{3-2y+\sqrt{(2y+1)^2}}{4}=1$. On the other hand, observing the identity $(2y+1)(2y+8x+9) = (2y+4x+5)^2 - 16(x+1)^2$ leads to $\frac{3-2y+\sqrt{(2y+1)(2y+8x+9)}}{4}<\frac{3-2y+(2y+4x+5)}{4}=x+2$.
Because $(2y+1)(2y+8x+9) = (2y+4x+5)^2 - 16(x+1)^2$, we have that 
$1=\frac{3-2y+(2y+1)}{4}<\psi(x,y)<\frac{3-2y+(2y+4x+5)}{4}=x+2$.
\end{proof}

\section{Deferred Proofs of Propositions in~\cref{sect:asymptotic}}\label{appendix:asymptotic}

In this appendix, we prove the estimates in~\cref{sect:asymptotic}.

\begin{proof}[Proof of~\cref{them-convergence1}] \label{proof-them-convergence1} First we prove  $\mathbf 1^Th_\algA^{(n)}+1\leqslant (n+1)^\varrho$ by induction. The result trivially holds for $n =0$ as $h_\algA^{(0)}=[~]^T$.  Let $r_n=\mathbf 1^Th_\algA^{(n)}$. From~\cref{prop-recur1} and~\cref{prop-eq-recur-r1}, we have recursion
\begin{equation}\label{convergence-r1}\begin{aligned}
    r_n&=\max_{0\leqslant k<n}\left\{ r_k + \varphi(r_k,r_{n-k-1})+r_{n-k-1}    \right\},
    \end{aligned}
\end{equation}
where $\varphi(x,y)=  (-x-y+\sqrt{(x+y+2)^2+4(x+1)(y+1)})/2$.
By induction, we assume  $r_k+1\leqslant (k+1)^\varrho$ for all $k<n$ and will show that $ r_n+1\leqslant (n+1)^\varrho$.

Adding $1$ on both sides of~\eqref{convergence-r1}, it then suffices to show that for $\forall 0\leqslant k<n$,
\begin{equation}\small \label{convergence-r2}
    \frac{(r_k+1)+(r_{n-k-1}+1)+\sqrt{(r_k+1+r_{n-k-1}+1)^2+4(r_k+1)(r_{n-k-1}+1)}}{2}\leqslant (n+1)^\varrho,
\end{equation}
where $r_k+1\leqslant (k+1)^\varrho$ and $r_{n-k-1}+1\leqslant (n-k)^\varrho$.
Let $w>0$ be the root of $(1+x/w)(1+y/w) = 2$, which is explicitly given by $w(x,y) = \frac{x+y+\sqrt{(x+y)^2+4xy}}{2}$. By~\cref{lemma:derivative} and noting that $(1+x/w)(1+y/w)$ is a decreasing function of $w$,
we obtain  $w(x,y)\leqslant (x^{\frac 1\varrho}+y^{\frac 1\varrho})^\varrho$. Setting $x = r_k +1\leqslant (k+1)^\varrho $ and $y = r_{n-k-1}+1\leqslant (n-k)^\varrho$, we  get $w(x,y)\leqslant (n+1)^\varrho$, which is just~\eqref{convergence-r2}. Thus the induction is complete.

Lastly, for $n=2^l-1$, $l\in\N_0$, by taking $k=n-k-1=2^{l-1}-1$ in~\eqref{convergence-r1}, we obtain that $r_n\geqslant 2r_k+\varphi(r_k,r_k)= (\sqrt 2+1)r_k+\sqrt 2$. Here we have used the fact $\varphi(x,x) = (\sqrt 2-1)x+\sqrt 2$. A simple induction on $l$ gives the lower bound $r_n +1\geqslant (\sqrt 2+1)^l = (n+1)^\varrho$ when $n=2^l-1$ for $l\in\mathbb N_0$. This, together with $r_n+1\leqslant (n+1)^\varrho$, indicates $r_n+1=(n+1)^\varrho$  when $n=2^l-1$, {$l\in\N_0$}.
\end{proof}

\begin{proof}[Proof of~\cref{them-convergence1.2}] \label{proof-them-convergence1.2}
As in the proof of~\cref{them-convergence1}, let $r_n=\mathbf 1^Th_\algA^{(n)}$. Then we still have~\eqref{convergence-r1} and thus
$r_n\geqslant r_k + \varphi(r_k,r_{n-k-1}) + r_{n-k-1}$ for $  0\leqslant k<n$,
where $\varphi(x,y)= (-x-y+\sqrt{(x+y+2)^2+4(x+1)(y+1)})/2$. Letting $k=n-1$, we immediately have $r_n > r_{n-1}$, implying $\{r_n\}$ is strictly increasing. Meanwhile, letting $k=\lfloor (n-1)/2\rfloor$ and $k'=n-k-1=\lceil (n-1)/2\rceil$ yields
\begin{equation}\label{convergence-r1.5}\begin{aligned}
    r_n&\geqslant r_k+\varphi(r_k,r_{k'})+r_{k'}
\stackrel{\text{(a)}}{\geqslant}r_k+\varphi(r_k,r_k)+r_k\stackrel{\text{(b)}}{=}(\sqrt 2+1)r_k +\sqrt 2,
    \end{aligned}
\end{equation}
where the inequality (a) uses the fact that $\varphi(x,y)+y$ is monotonically increasing with respect to $y$ and that $r_{k'}\geqslant r_k$, and equality (b) uses $\varphi(x,x) = (\sqrt 2-1)x+\sqrt 2$.

Noting that $k \ge n/2-1$, we get $k+2\geqslant (n+2)/2$. Hence~\eqref{convergence-r1.5} leads to
\begin{equation}\label{convergence-r1.52}
    \begin{aligned}
        \frac{r_n+1}{(n+2)^\varrho}
        \geqslant (\sqrt 2+1)\cdot \frac{r_k + 1}{(n+2)^\varrho}= \frac{r_k + 1}{((n+2)/2)^\varrho} \geqslant \frac{r_k + 1}{(k+2)^\varrho}.
    \end{aligned}
\end{equation}
When $2^l-1\leqslant n\leqslant 2^{l+1}-2$ and $l\geqslant 1$, for $k=\lfloor (n-1)/2\rfloor$ we have $2^{l-1}-1\leqslant k\leqslant 2^l - 2$. In this case, the right hand side of~\eqref{convergence-r1.52} is no smaller than $\nu_{l-1}$ by the definition of $\{\nu_l\}$. This implies that $(r_n+1)/(n+2)^\varrho\geqslant \nu_{l-1}$ holds for all $2^l-1\leqslant n\leqslant 2^{l+1}-2$ and $l\geqslant 1$. As a result, $\nu_l\geqslant \nu_{l-1}$, suggesting $\{\nu_l\}$ is non-decreasing.

In particular, $\nu_0$ is exactly the value $(\mathbf 1^Th_\algA^{(0)}+1)/(0+2)^\varrho = 1/(\sqrt 2+1)=\sqrt 2-1$. Also, for any $n\in\mathbb N_0$, there exists a nonnegative integer $l$ such that $2^l - 1\leqslant n\leqslant 2^{l+1}-2$. Hence we have $(\mathbf 1^Th_\algA^{(n)}+1)/(n+2)^\varrho\geqslant \nu_l\geqslant \nu_0=\sqrt 2 - 1$.
\end{proof}

\begin{proof}[Proof of~\cref{them-convergence2.1}]
Noting that the result holds for $n=0$ where $h_\algB^{(0)}=[~]^T\in\mathbb R^0$, we proceed by induction.
%Given iteration count $n\geqslant 1$, recall that~\cref{prop-eq-recur-r2} in~\cref{prop-recur2} gives the recursion that
%\begin{equation}\label{eq:hb-recur}
%    \begin{aligned}
%     \mathbf 1^Th_\algB^{(n)}
%     &=\max_{0\leqslant k<n}\left\{ \mathbf 1^Th_\algA^{(k)}+\psi(\mathbf 1^Th_\algA^{(k)},\mathbf 1^Th_\algB^{(n-k-1)})+\mathbf 1^Th_\algB^{(n-k-1)}\right\},
%    \end{aligned}
%\end{equation}
%where $\psi(x,y) = (3-2y+\sqrt{(2y+1)(2y+8x+9)})/4$.
%Changing the variable $z_n = 2(\mathbf 1^Th_\algB^{(n)})+1$ in~\eqref{eq:hb-recur} yields
Let $z_n = 2(\mathbf 1^Th_\algB^{(n)})+1$. Given iteration count $n\geqslant 1$,~\eqref{prop-eq-recur-r2}  implies
\begin{equation}\label{eq:zb-recur}
    z_n=\max_{0\leqslant k<n}\frac{4(\mathbf 1^Th_\algA^{(k)}+1)+z_{n-k-1}+\sqrt{z_{n-k-1}(z_{n-k-1} + 8(\mathbf 1^Th_\algA^{(k)}+1))}}{2}
\end{equation}
Assume the bound $z_k\leqslant \omega (k+1)^\varrho$ holds for each $k=0,1,\dotsc,n-1$ by induction. Also,~\cref{them-convergence1} claims that $\mathbf 1^Th_\algA^{(k)}+1\leqslant (k+1)^\varrho$. Letting $\mu=(n-k)/(n+1)$, we get $z_{n-k-1}\leq \omega \mu^\varrho(n+1)^\varrho$ and $\mathbf 1^Th_\algA^{(k)}+1\leqslant (1-\mu)^\varrho(n+1)^\varrho$.
Then~\eqref{eq:zb-recur} implies
% \begin{equation}\label{eq:hb-recur3}
%     z_n\leqslant \max_{0\leqslant k<n}\left\{
%     2(k+1)^\varrho + \frac12\omega(n-k)^\varrho + \sqrt{2\omega(k+1)^\varrho(n-k)^\varrho + \frac14\omega^2(n-k)^{2\varrho}}
%     \right\}.
% \end{equation}
% Dividing~\eqref{eq:hb-recur3} by $(n+1)^\varrho$ on both sides. Let $\mu = (n-k)/(n+1)$ and  thus $1-\mu = (k+1)/(n+1)$. Then~\eqref{eq:hb-recur3} gives
\begin{equation}\label{eq:hb-recur4}
   \frac{ z_n}{(n+1)^\varrho}\leqslant \max_{0\leqslant k<n}\left\{
    2(1-\mu)^\varrho + \frac12\omega\mu^\varrho + \sqrt{2\omega \mu^\varrho(1-\mu)^\varrho + \frac14\omega^2\mu^{2\varrho}}
    \right\}.
\end{equation}
% \begin{subarray}{c} 0\leqslant k<n\\  \mu = (n-k)/(n+1)\end{subarray}
% To complete the induction for $z_n\leqslant \omega(n+1)^\varrho$, it remains to show that the expression within the maximum operator of~\eqref{eq:hb-recur4} is upper-bounded by $\omega$ for all $\mu\in (0,1)$. %In fact, we introduce the following lemma, whose proof is deferred to Appendix \ref{appendix-B}.
Finally, since $\omega$ is the maximum of $\frac{2(1-\mu)^\varrho}{1-\mu^{\varrho/2}}$ over $\mu\in (0,1)$,  Lemma \ref{lemma:2.376373} in appendix implies that the right hand side of~\eqref{eq:hb-recur4} does not exceed $\omega$, regardless of the value of $\mu$. Thus we have $z_n\leqslant \omega(n+1)^\varrho$ and the induction is complete.
\end{proof}

\begin{proof}[Proof of~\cref{them-convergence2.2}]
Denote $z_n = 2(\mathbf 1^Th_\algB^{(n)})+1$ and it has $\liminf_{n\to\infty} z_n/n^\varrho=2\bar\nu_\algB$. Note that~\eqref{eq:zb-recur} in the proof of~\cref{them-convergence2.1} still holds.
Now we fix $\mu\in (0,1)$,   let $n\to\infty$ and set $k=\lfloor (1-\mu)(n-1)\rfloor$. 
Then {$\liminf_{n\to\infty}\mathbf 1^Th_\algA^{(k)}/n^\varrho\geq \bar\nu_\algA (1-\mu)^\varrho$ and $\liminf_{n\to\infty}z_{n-k-1}/n^\varrho\geq 2\bar\nu_\algB \mu^\varrho$}.
Dividing~\eqref{eq:zb-recur} by $n^\varrho$ on both sides and taking the limit inferior, we obtain the inequality
\begin{equation}\label{eq:liminf-hb-4}2\bar\nu_\algB\geqslant 2\bar\nu_\algA(1-\mu)^\varrho + \frac12(2\bar\nu_\algB) \mu^\varrho + \sqrt{2\bar\nu_\algA  (2\bar\nu_\algB)\cdot(\mu(1-\mu))^\varrho + \frac14(2\bar\nu_\algB)^2\mu^{2\varrho}}.%,~\forall \mu\in(0,1).
\end{equation}
% Dividing~\eqref{eq:liminf-hb-4} by $\bar\nu_\algA$ yields
% \begin{equation}\label{eq:liminf-hb-5}
%     \frac{2\bar\nu_\algB}{\bar\nu_\algA}\geqslant 2(1-\mu)^\varrho + \frac12\frac{2\bar\nu_\algB}{\bar\nu_\algA}\mu^\varrho + \sqrt{2\left(\frac{2\bar\nu_\algB}{\bar\nu_\algA}\right)\mu^\varrho (1-\mu)^\varrho + \frac14\left(\frac {2\bar\nu_\algB}{\bar\nu_\algA}\right)^2\mu^{2\varrho}},~\forall \mu\in(0,1).
% \end{equation}
Dividing~\eqref{eq:liminf-hb-4} by $\bar\nu_\algA$ and applying Lemma \ref{lemma:2.376373} in appendix,  we get $\frac{2\bar\nu_\algB}{\bar\nu_\algA}\geqslant \frac{2(1-\mu)^\varrho}{1-\mu^{\varrho/2}}$.
%~\forall \mu\in(0,1)$. 
This holds for arbitrary $\mu\in (0,1)$ and thus implies  $\frac{2\bar\nu_\algB}{\bar\nu_\algA}\ge \omega$ and 
$\bar\nu_\algB\ge \frac12\omega \bar\nu_\algA$.
\end{proof}

\section{Exact $h_\algB^{(n)}$}

% In~\cref{tab:Hb-exact} we present the stepsize  schedules $h_\algB^{(n)}$ from Definition \ref{def-alg2} for $1\leqslant n\leqslant 7$ in exact radicals. The numerical counterparts are displayed in~\cref{tab:Hb-numer}.
% Recalling that each $h_\algB^{(n)}$ takes the form  $h_\algB^{(n)} =[(h_\algA^{(k)})^T, \beta, (h_\algB^{(n-k-1)})^T]^T$ by \CPD, the intermediate stepsize $\beta$ has been highlighted in red to ensure clarity. We also remark that the stepsizes are not unique according to~\cref{sect:nonunique}.
% Note that $H_m[n]$ might not be unique.

 \begin{table}[http]
    \caption{Exact SSs of Definition~\ref{def-alg2} for $1\leqslant n\leqslant 7$ with the value $\beta$ of \CPD\  highlighted.}
    % Exact SSs of Definition~\ref{def-alg2} for $1\leqslant n\leqslant 7$. For some iteration counts, they are not unique. The intermediate stepsize $\beta$ of \CPD\ is highlighted.}
    \label{tab:Hb-exact}
    % \centering
    \renewcommand{\arraystretch}{2.5}\begin{tabular*}{\textwidth} {@{}c@{\hspace{0.5cm}}l}
    \bottomrule
        $n$ & $(h_\algB^{(n)})^T$ \\ \midrule
        1  & $\displaystyle \left[\mathbf{\frac32}\right]$ \\
        2 & $\displaystyle \left[\sqrt 2,\mathbf{\frac{3+\sqrt{9+8\sqrt 2}}{4}}\right]$\\
        3 & $\displaystyle \left[\sqrt 2,\mathbf{1+\sqrt 2},\frac32\right]$\\
        4 & $\displaystyle \left[ \sqrt 2,\frac{\sqrt{10+8\sqrt 2}-\sqrt 2}{2},\mathbf{\sqrt{3+\sqrt 2+\sqrt{10+8\sqrt 2}}},\frac32\right]$\\
        5 & $\displaystyle \left[\sqrt 2,2,\sqrt 2,\mathbf{\sqrt{7+4\sqrt 2}},\frac32\right]$\\
        6 & \scalebox{0.8}{$\displaystyle \left[\sqrt 2,2,\sqrt 2, \mathbf{
        \frac{3-4 \sqrt{2}-\sqrt{9+8 \sqrt{2}}+\sqrt{562+400 \sqrt{2}+\left(58+40 \sqrt{2}\right)\sqrt{9+8 \sqrt{2}} }}{8}
        },\sqrt 2,\frac{3+\sqrt{9+8\sqrt 2}}{4}\right]$}\\
        % -------------------------------------------------
        7 & \scalebox{0.8}{$ \left[\begin{aligned}
        & \sqrt 2,\frac{\sqrt{10+8\sqrt 2}-\sqrt 2}{2},\frac{-\sqrt{10+8\sqrt 2}-3\sqrt 2+2 \sqrt{19+14 \sqrt{2}+(7+4\sqrt{2})
   \sqrt{5+4 \sqrt{2}}}}{4},\sqrt 2,\\  &\quad \mathbf{
   1+\frac{1}{8} \left(5+4 \sqrt{2}+\sqrt{9+8 \sqrt{2}}\right) \left(\sqrt{1+\frac{a}{5+4 \sqrt{2}+\sqrt{9+8 \sqrt{2}}}}-1\right)
   }, \sqrt 2,\frac{3+\sqrt{9+8\sqrt 2}}{4}\end{aligned}\right] $}\\
   % -----------------------------------------
   % 8 & \!\scalebox{0.8}{
   % $\left[\begin{aligned}&\sqrt 2,2,\sqrt 2,\frac{ -2-3 \sqrt{2}+\sqrt{62+44\sqrt{2}}}{2},\sqrt 2,
   % \red{
   % 1+\frac{1}{8} \left(5+4 \sqrt{2}+\sqrt{9+8 \sqrt{2}}\right) \left(\sqrt{1+\frac{8 \left(4+3 \sqrt{2}+\sqrt{62+44\sqrt{2}}\right)}{5+4
   % \sqrt{2}+\sqrt{9+8 \sqrt{2}}}}-1\right)
   % },\sqrt 2,\frac{3+\sqrt{9+8\sqrt 2}}{4}\end{aligned}\right] $}\\
    \bottomrule
    \end{tabular*}
   {\footnotesize   For $n=7$, the value $a= 16+12 \sqrt{2}+4\sqrt{10+8 \sqrt{2}} +8 \sqrt{19+14 \sqrt{2}+(7+4 \sqrt{2})\sqrt{5+4 \sqrt{2}}}$.}
\end{table}

\bibliographystyle{siamplain}
\bibliography{references}
\end{document}